\documentclass[11pt, a4paper]{amsart}
\usepackage{amsfonts,amssymb,amsmath, amsthm}
\usepackage{mathrsfs}
\usepackage{lmodern}
\usepackage{qsymbols}
\usepackage{latexsym}
\usepackage[noadjust]{cite}
\usepackage{pdfsync}
\usepackage{bm}
\usepackage{enumitem}

\newtheorem{thm}{Theorem}[section]
\newtheorem{lem}[thm]{Lemma}

\theoremstyle{definition}
\newtheorem{defn}[thm]{Definition}
\newtheorem{rem}[thm]{Remark}

\numberwithin{equation}{section}
\usepackage[plainpages=false,pdfpagelabels,backref=page,citecolor=red]{hyperref}
\usepackage{xcolor}
\hypersetup{
colorlinks,
linkcolor={cyan!90!black},
citecolor={magenta},
urlcolor={green!40!black}
} 
\newcommand{\R}{\mathbb{R}}
\newcommand{\IC}{\mathbb{C}}

\newcommand{\IZ}{\mathbb{Z}}

\newcommand{\cH}{\mathcal{H}}

\newcommand{\loc}{\operatorname{loc}}
\renewcommand{\L}{\operatorname{L}} 
\newcommand{\Lloc}{\L_{\operatorname{loc}}} 
\newcommand{\C}{\operatorname{C}} 
\renewcommand{\H}{\operatorname{H}} 
\newcommand{\W}{\operatorname{W}}
\newcommand{\Hdot}{\dot{\H}\protect{\vphantom{H}}} 
\renewcommand{\S}{\mathrm{S}} 
\newcommand{\E}{\mathsf{E}} 
\newcommand{\reu}{{\mathbb{R}^{n+1}_+}}
\newcommand{\ree}{{\mathbb{R}^{n+1}}}
\renewcommand{\P}{P} 
\newcommand{\M}{M} 

\newcommand{\gradx}{\nabla_x}
\newcommand{\gradlamx}{\nabla_{\lambda,x}}
\renewcommand{\div}{\operatorname{div}}

\newcommand{\divx}{\div_x}
\newcommand{\dhalf}{D_t^{1/2}} 
\newcommand{\HT}{H_t} 
\newcommand{\pe}{\perp}
\newcommand{\pa}{\parallel}
\newcommand{\te}{\theta}
\newcommand{\Pfull}{\begin{bmatrix} 0 & \divx & -\dhalf \\ -\gradx & 0 & 0 \\ -\HT \dhalf & 0 & 0 \end{bmatrix}} 
\newcommand{\Mfull}{\begin{bmatrix} 1 & 0 & \vphantom{\dhalf}0 \\ 0 & {A}_{\pa \pa} & 0 \\ \vphantom{\dhalf}0& 0& 1 \end{bmatrix}} 
\newcommand{\NT}{\widetilde{N}_*} 

\newcommand{\e}{\mathrm{e}} 

\renewcommand{\i}{\mathrm{i}} 
\renewcommand{\d}{\, \mathrm{d}} 
\newcommand{\eps}{\varepsilon} 
\renewcommand\Re{\operatorname{Re}}

\newcommand{\cl}[1]{\overline{#1}} 
\DeclareMathOperator{\supp}{supp} 
\DeclareMathOperator{\id}{1} 
\DeclareMathOperator{\ran}{\mathsf{R}} 
\DeclareMathOperator{\dom}{\mathsf{D}} 
\DeclareMathOperator\Max{\mathcal{M}} 

\newcommand{\ta}{{\scriptscriptstyle \parallel}}
\newcommand{\no}{{\scriptscriptstyle\perp}}
\newcommand{\pd}{\partial}

\def\Xint#1{\mathchoice
{\XXint\displaystyle\textstyle{#1}}%
{\XXint\textstyle\scriptstyle{#1}}%
{\XXint\scriptstyle\scriptscriptstyle{#1}}%
{\XXint\scriptscriptstyle%
\scriptscriptstyle{#1}}%
\!\int}
\def\XXint#1#2#3{{\setbox0=\hbox{$#1{#2#3}{%
\int}$ }
\vcenter{\hbox{$#2#3$ }}\kern-.6\wd0}}
\def\barint{\,\Xint -} 
\def\bariint{\barint_{} \kern-.4em \barint}
\def\bariiint{\bariint_{} \kern-.4em \barint}
\renewcommand{\iint}{\int_{}\kern-.34em \int} 
\renewcommand{\iiint}{\iint_{}\kern-.34em \int} 
\newcommand{\I}{J}
\newcommand{\T}{I}
\newcommand{\tildeT}{II}
\newcommand{\hatT}{III}
\newcommand{\barT}{\T^{\tilde{\theta}}}

\title[The Dirichlet problem for second order parabolic operators]{The Dirichlet problem for second order\\ parabolic operators in divergence form}
\author{Pascal Auscher}
\address{Laboratoire de Math\'{e}matiques d'Orsay, Univ. Paris-Sud, CNRS, Universit\'{e} Paris-Saclay, 91405 Orsay, France \vspace{5pt}\newline\vspace{5pt}{\rm and}\newline Laboratoire Ami\'{e}nois de Math\'{e}matiques Fondamentales et Appliqu\'{e}es, UMR 7352 du CNRS, Universit\'{e} de Picardie-Jules Verne, 80039 Amiens, France}
\email{pascal.auscher@math.u-psud.fr}
\author{Moritz Egert}
\address{Laboratoire de Math\'{e}matiques d'Orsay, Univ.\ Paris-Sud, CNRS, Universit\'{e} Paris-Saclay, 91405 Orsay, France}
\email{moritz.egert@math.u-psud.fr}
\author{Kaj Nystr\"om}
\address{Department of Mathematics, Uppsala University, S-751 06 Uppsala, Sweden}
\email{kaj.nystrom@math.uu.se}
\thanks{The authors were partially supported by the ANR project ``Harmonic Analysis at its Boundaries'', ANR-12-BS01-0013. The second author was also supported by a public grant as part of the FMJH. The third author was supported by a grant from the G{\"o}ran Gustafsson Foundation for Research in Natural Sciences and Medicine.}

\date{\today}

\begin{document}
\begin{abstract}
We study parabolic operators $\cH = \partial_t-\div_{\lambda,x} A(x,t)\nabla_{\lambda,x}$ in the parabolic upper half space $\mathbb R^{n+2}_+=\{(\lambda,x,t):\ \lambda>0\}$. We assume that the coefficients are real, bounded, measurable, uniformly elliptic, but not necessarily symmetric. We prove that the associated parabolic measure is  absolutely continuous with respect to the surface measure on $\mathbb R^{n+1}$  in the sense defined by $A_\infty(\mathrm{d} x\d t)$.
Our argument also gives a simplified proof of the corresponding result for elliptic measure.
\end{abstract}

\maketitle

\section{Introduction and statement of main results}
    A classical result due to Dahlberg~\cite{Da} states in the context of Lipschitz domains that harmonic measure is  absolutely continuous with respect to surface measure, and that the Poisson kernel (its Radon-Nikodym derivative) satisfies a scale-invariant reverse H{\"o}lder inequality in $\L^2$. Equivalently, the Dirichlet problem with $\L^2$-data can be solved with $\L^2$-control of a non-tangential maximal function. Ever since Dahlberg's original work the study of elliptic measure has been a very active area of research and a number of fine results have been established, see~\cite{AAAHK, HKMP, KKPT} for recent accounts of the state of the art.

    In contrast to the study of elliptic measure, the fine properties of parabolic measure are considerably less understood. In~\cite{FSa} a parabolic version of Dahlberg's result was established for the heat equation in time-independent Lipschitz cylinders.  A major contribution in the study of boundary value problems and parabolic measure for the heat equation in time-dependent Lipschitz type domains was achieved in~\cite{LS, LM, HL}. In these papers the correct notion of time-dependent Lipschitz type cylinders, correct from the perspective of parabolic measure and parabolic
    layer potentials, was found.  In particular, in~\cite{LS, LM} the mutual absolute continuity of parabolic measure and surface measure and the $A_\infty$-property were established and in~\cite{HL} the authors obtained a version of Dahlberg's result for parabolic measure associated to the heat equation in time-dependent Lipschitz-type domains. In this context the properties of parabolic measures were further analyzed in the influential work~\cite{HL1}, parts of which have been simplified in
  ~\cite{NR}.

Very recently, there have been
advances in the theory of boundary value problems for second order parabolic equations (and systems) of the form
    \begin{eqnarray}\label{eq1}
    \mathcal{H}u:=\partial_tu-\div_{\lambda,x}A(x,t)\nabla_{\lambda,x}u=0,
    \end{eqnarray}
    in the  upper-half parabolic space $\mathbb R_+^{n+2}:=\{(\lambda,x,t)\in \mathbb R\times \mathbb R^n\times \mathbb R:\ \lambda>0\}$, $n\geq 1$, with boundary determined by $\lambda=0$, assuming only bounded, measurable, uniformly elliptic and complex coefficients. In~\cite{N1, CNS, N2}, the solvability for Dirichlet, regularity and Neumann problems with $\L^2$-data were established for the class of parabolic equations \eqref{eq1} under the additional assumptions that the elliptic part is also independent of the time variable $t$ and that it has either constant (complex) coefficients, real symmetric coefficients, or small perturbations thereof. Focusing on parabolic measure, a particular consequence of Theorem 1.3 in~\cite{CNS} is the generalization of~\cite{FSa} to equations of the form \eqref{eq1} but with $A$ real, symmetric and time-independent. This analysis was advanced further in~\cite{AEN}, where a first order strategy to study boundary value problems of parabolic systems with second order elliptic part in the upper half-space was developed. The outcome of~\cite{AEN} was the possibility to address arbitrary parabolic equations (and systems) as in \eqref{eq1} with coefficients depending also on time and on the transverse variable with additional transversal regularity.

    In this paper we advance the study of parabolic boundary value problems and parabolic measure even further. We consider parabolic equations as in \eqref{eq1}, assuming that the coefficients are real, bounded, measurable, uniformly elliptic, but not necessarily symmetric. We prove that the associated parabolic measure is  absolutely continuous with respect to the surface measure on
    $\mathbb R^{n+1}$ ($\mathrm{d} x\d t$) in the sense defined by the Muckenhoupt class $A_\infty(\mathrm{d} x\d t)$. As consequences, the associated Poisson kernel exists, satisfies a scale-invariant reverse H{\"o}lder inequality in $\L^p$ for some $p\in (1,\infty)$, and the Dirichlet problem with $\L^q$-data, $q$ being the index dual to $p$, can be solved with appropriate control of non-tangential maximal functions.  In particular, our main result, which  is new already in the case when $A$ is symmetric and time-dependent, gives a parabolic analogue of the main result in~\cite{HKMP} concerning elliptic measure. Our proof heavily relies on square function estimates and non-tangential estimates for parabolic operators with time-dependent coefficients that were only recently obtained by us in~\cite{AEN} as well as the reduction to a Carleson measure estimate proved in~\cite{DPP}. As we shall avoid the change of variables utilized in~\cite{HKMP}, this also gives a simpler and more direct proof of the $A_\infty$-property of elliptic measure.

  \subsection{The coefficients}    We assume that $A=A(x,t)=\{A_{i,j
}(x,t)\}_{i,j=0}^{n}$ is a real-valued $(n+1)\times (n+1)$-dimensional matrix, not necessarily symmetric,
    satisfying
    \begin{eqnarray}\label{eq2}
      \kappa|\xi|^2\leq \sum_{i,j=0}^{n}A_{i,j}(x,t)\xi_i\xi_j,\quad \ \ |A(x,t)\xi\cdot\zeta|\leq C|\xi||\zeta|,
    \end{eqnarray}
    for some $\kappa, C \in (0,\infty)$, which we refer to as the \emph{ellipticity constants} of $A$, and for all $\xi,\zeta\in \mathbb R^{n+1}$, $(x,t)\in\mathbb R^{n+1}$. Here, given $u=(u_0,...,u_n)$, $v=(v_0,...,v_n)\in\mathbb R^{n+1}$ we write $u\cdot v:=u_0v_0+...+ u_{n}v_{n}$.

\subsection{Weak solutions}
If $\Omega$ is an open subset of $\ree$, we let $\H^1(\Omega)=\W^{1,2}(\Omega)$ be the standard Sobolev space of {complex} valued functions $v$ defined on $\Omega$, such that $v$ and $\nabla v$ are in $\L^{2}(\Omega)$ and $\L^{2}(\Omega;\IC^n)$, respectively. A subscripted `$\loc$' will indicate that these conditions hold locally. A function $u$ is called \emph{weak solution} to the equation $\mathcal{H} u=0$ on $\reu\times \R$ if it satisfies
$u\in \Lloc^2(\R; \W^{1,2}_{\loc}(\reu))$ and
\begin{align*}
  \int_\R\iint_{\reu} A\nabla_{\lambda,x} u\cdot{\nabla_{\lambda,x} \phi}\, \d x \d t \d\lambda - \int_{\R} \iint_{\reu} u \cdot {\partial_{t}\phi}\, \d x \d t \d\lambda =0
\end{align*}
for all $\phi\in \C_0^\infty(\R^{n+2}_+)$.

\subsection{Parabolic measure}\label{parabolic measure}
Given $(x,t)\in \R^{n+1}$ and $r>0$ we let $Q=Q_r(x):=B(x,r)\subset\mathbb R^n$ be the standard Euclidean ball centered at $x$ and of radius $r$, and we let $I=I_r(t):=(t-r^2,t+r^2)$. We let $\Delta=\Delta_r(x,t)=Q_r(x)\times I_r(t)$ and write $\ell(\Delta):=r$. We will use the convention that $cQ$ and $cI$ denote the dilates of balls and intervals, respectively, keeping the center fixed and dilating the radius by $c$ and we let $c\Delta:= cQ\times c^2I$.

Given $A$ real, satisfying \eqref{eq2}, and $f$ continuous and compactly supported in $\ree$, there exists a  unique (weak) solution $u$ to the continuous Dirichlet problem  $\mathcal{H}u=(\partial_t-\div_{\lambda,x} A(x,t)\nabla_{\lambda,x})u=0$ in $\mathbb R^{n+2}_+$, $u$ continuous  in $\overline{\R^{n+2}_{+}}$ and $u(0,x,t)=f(x,t)$ whenever $(x,t)\in \mathbb R^{n+1}$. Indeed, assume $f\geq 0$ and let $u_k$, $k\geq 1$, be the unique weak solution to $\mathcal{H}u=0$ in $\Omega_k:=(0,k)\times \Delta_{k}(0,0)$, with boundary values $f(x,t)\psi(||(x,t)||/k)$ on $\Delta_{k}(0,0)$, and zero otherwise. Here, $||(x,t)||:=|x|+|t|^{1/2}$ and $\psi$ is a continuous decreasing function on $[0,\infty)$ such that $0\leq\psi\leq 1$, $\psi(r) = 1$ for $0\leq r\leq 1/2$, and $\psi(r) = 0$ for $r > 3/4$. Then $0\leq u_k\leq u_{k+1}\leq ||f||_\infty$ in $\Omega_k$ and one can deduce, by the maximum principle and the Harnack inequality, see \cite{N} for these estimates, that
$$\sup_{\Omega_l}|u_k-u_j|\leq c(u_k-u_j)(l,0,4l^2), \qquad \mbox{if $k>j\gg l$}.$$
In particular, $u$ can be constructed as the monotone and uniform limit of $\{u_k\}$ as $k\to \infty$  on the closure of $\Omega_l$ for each $l\geq 1$. Uniqueness follows from the maximum principle. Furthermore, by the maximum principle and the Riesz representation theorem we deduce $$u(\lambda,x,t)=\iint_{\mathbb R^{n+1}}f(y,s)\, \d\omega(\lambda,x,t,y,s), \qquad \mbox{for all $(\lambda, x,t)\in \mathbb R^{n+2}_+$},$$
where $\{\omega(\lambda,x,t,\cdot):\ (\lambda,x,t)\in \mathbb R^{n+2}_+\}$ is a  family of regular Borel measures on $\mathbb R^{n+1}$ and we refer to $\omega(\lambda,x,t,\cdot)$ as  $\mathcal{H}$-parabolic measure, or simply \emph{parabolic measure} (at $(\lambda,x,t)$).

Given $r>0$ and $(x_0,t_0)\in \mathbb R^{n+1}$ we let
\begin{eqnarray*}
A_{r}^+(x_0,t_0):=(4r,x_0,t_0+16r^2).
\end{eqnarray*}
Assume that $A$ satisfies \eqref{eq2}. Then parabolic measure is  a doubling measure in the sense that there exists a constant $c$,  $1\leq c<\infty$, depending only on $n$ and the ellipticity constants such that the following is true. Let $(x_0,t_0)\in \mathbb R^{n+1}$, $0<r_0<\infty$, $\Delta_0:=\Delta_{r_0}(x_0,t_0)$.  Then
\begin{eqnarray*}
\omega\bigl (A_{4r_0}^+(x_0,t_0), 2\Delta\bigr )\leq
c\omega\bigl (A_{4r_0}^+(x_0,t_0), \Delta\bigr )
\end{eqnarray*}
whenever  $\Delta\subset 4\Delta_0$. We refer to~\cite{FS}, \cite{FSY} and \cite{N} for details. The doubling property of parabolic measure serves as a starting point for further investigation. In this paper we are interested in scale invariant quantitative version of absolute continuity of parabolic measure with respect to the measure $\mathrm{d} x\d t$ on $\mathbb R^{n+1}$. Given a set $E\subset \mathbb R^{n+1}$ we let $|E|$ denote the Lebesgue measure of $E$.

\begin{defn}\label{Ainfty} Let $(x_0,t_0)\in \mathbb R^{n+1}$, $0<r_0<\infty$, $\Delta_0:=\Delta_{r_0}(x_0,t_0)$. We say that parabolic measure associated to  $\mathcal{H}= \partial_t-\div_{\lambda,x} A(x,t)\nabla_{\lambda,x}$ at $A_{4r_0}^+(x_0,t_0)$
is in $A_\infty(\Delta_0,\d x\d t)$ if for every
$\varepsilon > 0$ there exists $\delta = \delta(\varepsilon)>0$ such that if $E\subset \Delta$ for some $\Delta\subset\Delta_0$, then
$$\frac {\omega\bigl (A_{4r_0}^+(x_0,t_0),E\bigr )}{\omega\bigl (A_{4r_0}^+(x_0,t_0),\Delta\bigr )}<\delta \quad \Longrightarrow \quad \frac{ | E | }{ |\Delta|}<\varepsilon.$$
Parabolic measure $\omega$ belongs to $A_\infty(\mathrm{d} x\d t)$ if $\omega\bigl (A_{4r_0}^+(x_0,t_0),\cdot\bigr )\in A_\infty(\Delta_0,\d x\d t)$ for all $\Delta_0$ as above and with uniform constants.
\end{defn}

If $\omega$ belongs to $A_\infty(\mathrm{d} x\d t)$, then  $\omega(A_{4r_0}^+(x_0,t_0),\cdot)$ and $\mathrm{d} x\d t$ are mutually absolutely continuous and hence one can write
$$\d\omega\bigl (A_{4r_0}^+(x_0,t_0),x,t\bigr )=K\bigl (A_{4r_0}^+(x_0,t_0),x,t\bigr )\d x\d t.$$ We refer to $K\bigl (A_{4r_0}^+(x_0,t_0),x,t\bigr )$ as the associated \emph{Poisson kernel} (at $A_{4r_0}^+(x_0,t_0)$).
\begin{defn}\label{reverse}
For $p\in (1,\infty)$ we say that $\omega$ belongs to the reverse H{\"o}lder class $B_p(\mathrm{d} x\d t)$ if there exists a constant $c$, $1\leq c<\infty$, such that for all $\Delta_0:=\Delta_{r_0}(x_0,t_0)$ the Poisson kernel $K\bigl (A_{4r_0}^+(x_0,t_0),\cdot\bigr )$
satisfies the reverse H{\"o}lder inequality
$$\biggl (\bariint_{\Delta} (K\bigl (A_{4r_0}^+(x_0,t_0),x,t\bigr ))^p\, \d x\d t\biggr )^{1/p}\leq c\bariint_{\Delta} K\bigl (A_{4r_0}^+(x_0,t_0),x,t\bigr )\, \d x\d t$$
whenever $\Delta\subset\Delta_0$.
\end{defn}

Note that as parabolic measure has the doubling property the statement that parabolic measure $\omega$ belongs to $A_\infty(\mathrm{d} x\d t)$ has several
equivalent formulations. Furthermore, $A_\infty(\mathrm{d} x\d t)= \bigcup_{p>1}B_p(\mathrm{d} x\d t)$.  We refer to~\cite{CF} for more on $A_\infty$. For $(x,t)\in \ree$, and a function $F$, we define the non-tangential maximal function
\begin{eqnarray}
\label{eq:NTmaxDef-}
N_\ast F(x,t)= \sup_{\lambda>0}\sup_{\Lambda \times Q \times I} |F(\mu,y,s) |,
\end{eqnarray}
where $\Lambda=(\lambda/2, \lambda)$, $Q=B(x, \lambda)$ and $I=(t-\lambda^2, t+\lambda^2)$. Given $(x_0,t_0)\in \mathbb R^{n+1}$, $\eta>0$, we also introduce the  parabolic cone
  \begin{eqnarray}
\label{cone}\Gamma^\eta(x_0,t_0):=\{(\lambda, x,t)\in\mathbb R^{n+2}_+:\ ||(x-x_0,t-t_0)||<\eta\lambda\}.
\end{eqnarray}

\begin{defn}\label{D_q} Let $q\in (1,\infty)$. We say that the Dirichlet problem for $\mathcal{H}$ in $\mathbb R^{n+2}_+$ with data in $\L^q(\mathbb R^{n+1})$, $D_q$ for short, is solvable if the following holds.  Given $f\in \L^q(\mathbb R^{n+1})$ then there exists a weak solution $u$ such that
	      \begin{align*}
	      \mathcal{H}u &= 0 &&\mbox{ \hspace{-80pt}in $\mathbb R^{n+2}_+$},\notag\\
	      \lim_{\lambda\to 0}u(\lambda,\cdot,\cdot)&=f(\cdot,\cdot) &&\mbox{ \hspace{-80pt}in }\L^q(\mathbb R^{n+1})\mbox{ and n.t.},\notag\\
	      || N_\ast u||_q&<\infty.&&
	      \end{align*}
	      Here, n.t.\ is short for non-tangentially and means $u(\lambda,x,t)\to f(x_0,t_0)$ for almost every $(x_0,t_0)\in\mathbb R^{n+1}$ as $(\lambda,x,t)\to (x_0,t_0)$ through the parabolic cone $\Gamma^\eta(x_0,t_0)$ for some $\eta>0$. Furthermore, we say that $D_q$ \emph{{uniquely} solvable} if $D_q$ is solvable and if the solution is unique.
	      \end{defn}

Assume that parabolic measure $\omega$ belongs to $A_\infty(\mathrm{d} x\d t)$ and, in particular, that $\omega$ belongs to $B_p(\mathrm{d} x\d t)$ for some $p\in (1,\infty)$.  The latter is equivalent to the statement that $D_q$ for $\mathcal{H}$ is solvable, $q$ being the dual index to $p$, see for example Theorem 6.2 in \cite{N}. While the results in \cite{N} are derived under the assumption of symmetric coefficients, the lemmas underlying the proof of Theorem 6.2 in \cite{N} do not rely on this assumption.

\begin{rem}\label{RelationRevH-Dp} Concerning $D_q$ being \emph{{uniquely} solvable}, establishing a criteria for this  in terms of parabolic measure is more complicated and forces one to also consider the adjoint parabolic measure.  The adjoint parabolic measure $\omega^\ast$ is the parabolic measure associate to $\mathcal{H}^\ast:= -\partial_t-\div_{\lambda,x} A^\ast(x,t)\nabla_{\lambda,x}$, $A^\ast$ being the transpose of $A$. Definition \ref{Ainfty}  and Definition \ref{reverse} for $\omega^\ast$ are as stated but with the point $A_{4r_0}^+(x_0,t_0)$ replace by $A_{4r_0}^-(x_0,t_0)$ where $A_{r}^-(x_0,t_0):=(4r,x_0,t_0-16r^2)$ for $r>0$. We claim that one can prove that if
$\omega$ belongs to $B_p^{\mathcal{H}}(\mathrm{d} x\d t)$ and $\omega^\ast$ belongs to $B_p^{\mathcal{H}^\ast}(\mathrm{d} x\d t)$, then $D_q$ for $\mathcal{H}$ is uniquely solvable, $q$ still being the dual index to $p$. The assumption that $\omega^\ast$ belongs to $B_p^{\mathcal{H}^\ast}(\mathrm{d} x\d t)$ is used to conclude the uniqueness. The proof of the claim is akin to the elliptic argument in Theorem 1.7.7 in \cite{K}.
\end{rem}

\subsection{Statement of the main result} The following theorem is our main result.

\begin{thm}\label{Ainfty+} Assume that $A$ satisfies \eqref{eq2}. Then parabolic measure $\omega$ belongs to $A_\infty(\mathrm{d} x\d t)$ with constants depending only $n$ and the ellipticity constants. In particular, there exists $p\in (1,\infty)$ such that $\omega$ belongs to the reverse H{\"o}lder class $B_p(\mathrm{d} x\d t)$  with $p$ and the constant in the reverse H{\"o}lder inequality depending only $n$ and the ellipticity constants. Equivalently, $D_q$, where $q$ is the index dual to $p$, is solvable.
\end{thm}

Theorem~\ref{Ainfty+} is new and gives the parabolic counterpart of the corresponding recent result for elliptic measure obtained in~\cite{HKMP}, with a simplified argument compared to~\cite{HKMP}. As mentioned before, Theorem~\ref{Ainfty+} is new even in the case when $A$ is symmetric and time-dependent. Note that in~\cite{JK} the result of Dahlberg was proved for elliptic measure associated to the elliptic counterpart of \eqref{eq1} with symmetric $A$, that is, in this case the associated Poisson kernel exists and belongs to $B_2$. In contrast, in the parabolic case it is not clear if such a result holds true if we allow for time-dependent coefficients (the case of time-independent coefficients was treated in~\cite{CNS} and does give $B_{2}$).

Theorem~\ref{Ainfty+} generalizes immediately to the setting of time-independent Lipschitz domains in the following sense. Consider the domain $\{(x_0,x,t):\ x_0>\varphi(x)\}$ above the graph of the time-independent Lipschitz function $\varphi$ and consider the equation
$$\partial_tu-\div_{x_0,x}A(x,t)\nabla_{x_0,x}u=0$$ in this domain. Using the simple change of variables $(\lambda, x,t)\mapsto (\lambda+\varphi(x),x,t)$, this equation is equivalent to an equation in the upper parabolic half space to which Theorem~\ref{Ainfty+} applies. In contrast, this argument does not apply to a time-dependent domain of the form $\{(x_0,x,t):\ x_0>\varphi(x,t)\}$ as the change of variables  $(\lambda, x,t)\mapsto (\lambda+\varphi(x,t),x,t)$ with $\varphi$ Lipschitz in both $x$ and $t$ destroys the structure of the equations studied here. If $\varphi$ is only Lipschitz with respect to the parabolic metric, that is, Lipschitz continuous in $x$ and $1/2$-H{\"o}lder continuous in $t$, then more elaborate changes of variables have to be employed but this changes the nature of the assumption on the coefficients, see~\cite{HL1} for details.

\subsection{Outline of the proof of Theorem~\ref{Ainfty+}}\label{proof} The proof consists of three parts: a reduction to a Carleson measure estimate, the construction of a particular set $F$, and the proof of the Carleson measure estimate by partial integration. These three parts have four sources of insights~\cite{KKPT, HKMP, AEN, DPP}. In general, $c$ will denote a generic constant, not necessarily the same at each instance, which, unless otherwise stated, only depends on $n$ and the ellipticity constants. We often write $c_1\lesssim c_2$ when we mean that $c_1/c_2$ is bounded by a constant depending only $n$ and the ellipticity constants.

\subsubsection*{Reduction to a Carleson measure estimate}
The key insight in~\cite{KKPT} is that the $A_\infty$-property of elliptic measure follows once a certain Carleson measure condition is verified. More recently, this idea has also been implemented in the parabolic context: On pp.1172--1175 in~\cite{DPP} it is shown that in order to conclude $\omega \in A_\infty(\mathrm{d} x \d t)$ it suffices to prove the following result, which we state here as our second main theorem.

\begin{thm}\label{Carleson for bounded}
Let $S \subset \ree$ be a bounded Borel set and let $u(\lambda,x,t):=\omega(\lambda,x,t,S)$ be the corresponding weak solution to \eqref{eq1} created by the $\cH$-parabolic measure $\omega$. Then $u$ satisfies the following Carleson measure estimate: for all parabolic cubes $\Delta \subset \ree$,
\begin{eqnarray}\label{Carl1-}
\int_0^{\ell(\Delta)}\iint_{\Delta} |\nabla_{\lambda,x } u |^2\,  \lambda\, \d x\d t\d\lambda\lesssim |\Delta|.
\end{eqnarray}
\end{thm}

\begin{rem}\label{Reduction to smooth}
Theorem~\ref{Carleson for bounded} is \emph{a priori} equivalent to the statement that \eqref{Carl1-} holds for all parabolic cubes whenever $u$ is the unique solution to the continuous Dirichlet problem for $\cH u = 0$ with continuous compactly supported boundary data $f$ satisfying $|f| \leq 1$, see Remark~5 in \cite{DPP}. Note that in this case $|u| \leq 1$ by the maximum principle. This reformulation has the advantage that it allows one to assume that $A$ is smooth as long as all bounds depend on $A$ only through its ellipticity constants, see p.~20 in \cite{HL1} for this type of reduction.
\end{rem}

Based on Remark \ref{Reduction to smooth} we can assume \emph{qualitatively} that $A$ is smooth and we are left with the task of proving the Carleson measure estimate \eqref{Carl1-} if $u$ is \emph{any} weak solution to \eqref{eq1} bounded by $|u| \leq 1$. The fact that $u$ could be chosen continuous up to the boundary will not enter the argument. 

As a first reduction step we claim that instead of \eqref{Carl1-} it suffices to prove for all parabolic cubes $\Delta$,
\begin{eqnarray}\label{Carl1}
\int_0^{\ell(\Delta)}\iint_{\Delta} |\partial_{\lambda } u |^2\,  \lambda\, \d x\d t\d\lambda\lesssim |\Delta|.
\end{eqnarray}
To see this, we truncate the integral on the left at $2\eps>0$ and pick a piecewise linear function $\eta= \eta(\lambda)$, equal to $1$ on $(2\eps,\ell(\Delta))$ and equal to $0$ on $(0,\eps)$ and $(2\ell(\Delta), \infty)$. In particular, $|\lambda \partial_\lambda \eta| \leq 2$. Integration by parts in $\lambda$ on the term $\eta |\nabla_{\lambda,x}u|^2 \lambda$ then leads to
\begin{align*}
\int_{2 \eps}^{\ell(\Delta)}\iint_{\Delta} |\nabla_{\lambda,x } u |^2\,  \lambda\, \d x\d t\d\lambda
&\leq 2\int_{\eps}^{2 \eps} \iint_{\Delta} |\nabla_{\lambda,x } u |^2 \, \lambda \, \d x\d t \d \lambda
+ 2\int_{\ell(\Delta)}^{2\ell(\Delta)} \iint_{\Delta} |\nabla_{\lambda,x } u |^2 \, \lambda \, \d x\d t \d \lambda \\
&\quad + \frac{1}{2} \int_{2\eps}^{\ell(\Delta)} \iint_\Delta |\nabla_{\lambda,x}u|^2 \, \lambda + \frac{1}{2} \int_{2\eps}^{\ell(\Delta)} \iint_\Delta |\nabla_{\lambda,x} \partial_\lambda u|^2 \, \lambda^3 \, \d x \d t \d \lambda,
\end{align*}
where the third and fourth term arise from bounding $\frac{\lambda^2}{2} \partial_\lambda |\nabla_{\lambda,x} u|^2$ via Young's inequality. 
The standard Caccioppoli inequality (see Lemma~\ref{lem:Caccioppoli} below) along with the uniform bound $|u| \leq 1$ allows us to control the first two integrals on the right-hand side by $C |\Delta|$, where $C$ depends on $n$ and the ellipticity constants. The third integral is finite and can be absorbed into the left-hand side. Finally, for the fourth integral we use that $\partial_\lambda u$ is a solution to $\mathcal{H} u=0$ as well ($A$ is independent of $\lambda$) and apply Caccioppoli's inequality on parabolic Whitney cubes covering $(2 \eps, \ell(\Delta)) \times \Delta$. In total, we get
\begin{align*} 
 \frac{1}{2} \int_{2 \eps}^{\ell(\Delta)}\iint_{\Delta} |\nabla_{\lambda,x } u |^2\,  \lambda\, \d x\d t\d\lambda
\lesssim |\Delta|
+ \int_{\eps}^{2 \ell(\Delta)}\iint_{2\Delta} |\partial_\lambda u |^2\,  \lambda \, \d x\d t\d\lambda.
\end{align*}
Passing to the limit $\eps \to 0$, we see that having \eqref{Carl1} for all parabolic cubes is sufficient for having \eqref{Carl1-} for all parabolic cubes. Hence, we can concentrate on \eqref{Carl1}.

Furthermore, as our equations have real and uniformly elliptic coefficients, the solution $\partial_\lambda u$ satisfies De Giorgi--Moser--Nash estimates, see for example Lemmas~3.3 and 3.4 in~\cite{HL1} or \cite{A}.  From a John--Nirenberg Lemma for Carleson measures, see Lemma 2.14 in
	      \cite{AHLeT}, it follows that for \eqref{Carl1} it is sufficient to prove that the following holds: For each parabolic cube $\Delta\subset\mathbb R^{n+1}$, $r:=\ell(\Delta)$,  there is a Borel set $F\subset 16 \Delta$  with $|\Delta|\lesssim |F|$, such that
		\begin{eqnarray}\label{Carl2}
	\int_0^{r}\iint_{F}|\partial_\lambda u|^2\, \lambda \d x\d t\d\lambda\lesssim |\Delta|.
	      \end{eqnarray}
Indeed, let $H_\lambda(x,t):=|\partial_\lambda u(\lambda,x,t)|^2\lambda^2$. Again from De Giorgi--Moser--Nash estimates we can infer $0\leq H_\lambda(x,t)\lesssim 1$ in $(0,\infty) \times \R^{n+1}$
and 
\begin{align*}
 |H_\lambda(x,t)-H_\lambda(x',t')|\lesssim \frac{||(x-x',t-t')||^{\alpha}}{\lambda^\alpha},
\end{align*}
for some $\alpha=\alpha(n,\kappa,C)>0$ whenever $(\lambda,x,t), (\lambda,x',t') \in(0, \infty)\times \R^{n+1}$.
Hence, we are in the setup of Lemma 2.14 in \cite{AHLeT} with parabolic scaling. Its proof can then be readily adapted to justify the reduction in \eqref{Carl2}. 

Note that in \eqref{Carl2} the set $F$ is a degree of freedom subject to the restrictions. This completes our reduction to a Carleson measure estimate. To avoid duplication with~\cite{DPP} and for the sake of brevity, we will not give more details concerning these facts. Instead we will simply prove Theorem~\ref{Carleson for bounded} and Theorem~\ref{Ainfty+} by verifying \eqref{Carl2} for a properly constructed set $F$ and this is the main contribution of the paper.

\subsubsection*{Construction of the set $F$}
In the context of elliptic measure the freedom of having a set $F \subset \Delta$ at one's disposal in \eqref{Carl2} was cleverly brought into play in~\cite{HKMP} via an adapted Hodge decomposition. Inspired by this, we look for a parabolic Hodge decomposition. To this end, we split the coefficient matrix $A$ as
\begin{eqnarray}
\label{eq:A}
A(x,t)= \begin{bmatrix} A_{\no\no}(x,t) & A_{\no\ta}(x,t)\\ A_{\ta\no}(x,t) & A_{\ta\ta}(x,t) \end{bmatrix}.
\end{eqnarray}
Then $A_{\no\ta}$ is an $n$-dimensional row vector and $A_{\ta\no}$ is an $n$-dimensional column vector. We have a similar decomposition of $A^\ast$, which is the transpose of $A$ since $A$ has real coefficients.

Introduce the parabolic operator $\mathcal{H}_{\pa}:=\partial_t-\div_{x}A_{\pa\pa}\nabla_{x}$ and its adjoint  $\mathcal{H}_{\pa}^\ast:=-\partial_t-\div_{x}A_{\pa\pa}^\ast\nabla_{x}$ on $\R^{n+1}$.
Let us recall that $\cH_\pa$ and $\cH_\pa^\ast$ admit the following \emph{hidden coercivity} used systematically in~\cite{N1, CNS, N2,AEN}. In fact, it appeared already in \cite{Kaplan}.  First, we define the homogeneous \emph{energy space} $\dot{\E}(\ree)$ by taking the closure of test functions $v \in \C_0^\infty(\ree)$ with respect to the norm
\begin{align*}
\|v\|_{\dot{\E}(\ree)}^2 := \iint_{\ree} |\gradx v|^2 + |\dhalf v|^2 \, \d x \d t
\end{align*}
{and identifying functions that differ only by a constant.} Here, the half-order $t$-derivative $\dhalf$ is defined via the Fourier symbol $\i |\tau|^{1/2}$. This closure can be realized in $\L^2(\ree) + \L^\infty(\ree)$ and modulo constants $\dot{\E}(\ree)$ becomes a Hilbert space, see for example Section~3.2 in~\cite{AEN}. The corresponding inhomogeneous energy space $\E(\ree) = \dot{\E}(\ree) \cap \L^2(\ree)$ is equipped with the obvious Hilbertian norm. Denoting by $\HT$ the Hilbert transform with respect to the $t$-variable, we can factorize $\partial_t = \dhalf \HT \dhalf$ and this in turn allows us to define $\cH_\pa$ as a bounded operator from $\dot{\E}(\ree)$ into its (anti)-dual $\dot{\E}(\ree)^*$ via
\begin{align}
\label{hidden coercivity}
(\cH_\pa u)(v) := \iint_{\ree} \dhalf u \cdot \cl{\HT \dhalf v} + A_{\pa\pa} \gradx u \cdot \cl{\gradx v} \, \d x \d t.
\end{align}
The hidden coercivity of the sesquilinear form on the right-hand side now pays for this operator being invertible with operator norm depending only on $n$ and the ellipticity constants of $A_{\pa \pa}$, see Theorem~1 in \cite{Kaplan} or Lemma~5.9 in~\cite{CNS}.
An analogous construction applies to $\cH_\pa^*$. Considering a parabolic cube $ \Delta=\Delta_r\subset\mathbb R^{n+1}$, we let $\chi_{8\Delta}=\chi_{8\Delta}(x,t)$ be a smooth cut off for $8\Delta$ which is $1$ on $8\Delta$, vanishes outside of $16\Delta$ and satisfies $r|\nabla_{x} \chi_{8\Delta}|+r^2|\partial_t\chi_{8\Delta}|\leq c$. Then, there exist $\varphi, \tilde\varphi \in \dot{\E}(\ree)$ solving
		  \begin{eqnarray}\label{eq7+int}
	  \mathcal{H}_{\pa}^\ast\varphi=\div_{x}({A_{\no\ta}}\chi_{8\Delta}),\qquad \mathcal{H}_{\pa}\tilde\varphi=\div_{x}({A_{\ta\no}}\chi_{8\Delta}),
	      \end{eqnarray}
and satisfying the \emph{a priori} estimates
	      \begin{eqnarray}\label{ra1int}
	      \begin{split}
&&\iint_{\mathbb R^{n+1}} |\nabla_{x}  \varphi|^2+|\HT \dhalf\varphi|^2 \, \d x\d t\lesssim \iint_{16\Delta}|A_{\no\ta}|^2\, \d x\d t\lesssim|\Delta|, \\
&&\iint_{\mathbb R^{n+1}} |\nabla_{x}  \tilde\varphi|^2+|\HT \dhalf\tilde\varphi|^2 \, \d x\d t\lesssim \iint_{16\Delta}|A_{\ta\no}|^2\, \d x\d t\lesssim| \Delta|.
\end{split}
\end{eqnarray}
We refer to $\varphi$ and $\tilde \varphi$ as parabolic hodge decompositions of the vector fields ${A_{\no\ta}}\chi_{8\Delta}$ and ${A_{\ta\no}}\chi_{8\Delta}$, respectively. These decompositions give representations of the vector fields ${A_{\no\ta}}\chi_{8\Delta}$, ${A_{\ta\no}}\chi_{8\Delta}$ adapted to the operators $\mathcal{H}_{\pa}^\ast$, $\mathcal{H}_{\pa}$, representations which combined with the \emph{a priori} estimates in  \eqref{ra1int} allow us to make use of the powerful toolbox behind the solution of the parabolic Kato problem in \cite{AEN}. Note that as we can undo the factorization of $\partial_t$ leading to \eqref{hidden coercivity} if $v$ is a test function, \eqref{eq7+int} holds \emph{a fortiori} in the usual weak sense.
More in the spirit of operator theory, Lemma~4 in~\cite{AE} shows that the part of $\cH_\pa$ in $\L^2(\ree)$ with maximal domain
\begin{align*}
\dom(\cH_\pa) = \{u \in \E(\ree) : \cH_\pa u \in \L^2(\ree) \}
\end{align*}
is \emph{maximal accretive}, that is, for every $\mu \in \IC$ with $\Re \mu > 0$ the operator $\mu + \cH_\pa$ is invertible and $\|(\mu + \cH_\pa)^{-1}\|_{\L^2 \to \L^2} \leq (\Re \mu)^{-1}$ holds. The recent resolution of the Kato problem for parabolic operators identifies the domain of its unique maximal accretive square root as $\dom(\cH_\pa^{1/2}) = \E(\ree)$ with a homogeneous estimate
\begin{align*}
\|\cH_\pa^{1/2} v\|_2 \sim \|\gradx v\|_2 + \|\HT \dhalf v\|_2 \qquad \mbox{for $v \in \E(\ree)$},
\end{align*}
see Theorem~2.6 in~\cite{AEN}. Thus, writing
\begin{align*}
(\mu + \cH_\pa)^{-1}v = \cH_\pa^{-1/2} (\mu + \cH_\pa)^{-1} \cH_\pa^{1/2}v,
\end{align*}
we can extend $(\mu + \cH_\pa)^{-1}$ by density from $\E(\ree)$ to a bounded and invertible operator on $\dot{\E}(\ree)$. Again we also have the analogous results for $\cH_\pa^*$. In particular, for $m$ a natural number and $\lambda > 0$ we can introduce the higher order resolvents of $\varphi, \tilde \varphi$,
      \begin{eqnarray}\label{paris3}
	\P_{\lambda}^\ast\varphi:=(1+\lambda^2\mathcal{H}_{\pa}^\ast)^{-m}\varphi,\qquad \P_{\lambda}\tilde\varphi:=(1+\lambda^2\mathcal{H}_{\pa})^{-m}\tilde\varphi,
      \end{eqnarray}
within the homogeneous energy space $\dot{\E}(\ree)$. In the further course we will fix $m$ large enough (without trying to get optimal values) to have a number of estimates at our disposal. In fact, as can be seen from the proof of 
Lemma~\ref{Est1} below,  $m=n+1$ is sufficient for our purposes as this allows us to prove pointwise estimates of certain kernels needed in the proof of non-tangential maximal estimates of
$\partial_\lambda \P_{\lambda}^\ast\varphi$ and $\partial_\lambda \P_{\lambda}\tilde\varphi$. Coming back to the actual construction of $F$, we also introduce the parabolic maximal differential operator
\begin{eqnarray}\label{lipll+}
\mathbb D v(x,t)&:=&\sup_{\varrho>0} \bariint_{\Delta_\varrho(x,t)} \frac {|v(x,t)-v(y,s)|}{||(x-y,t-s)||} \, \d y\d s, \quad v \in \dot{\E}(\ree),
\end{eqnarray}
which maps boundedly into $\L^2(\ree)$ as we shall prove later on in Lemma~\ref{Quotient}. Here, $\|\cdot \|$ indicates again the parabolic distance. In particular, \eqref{ra1int} implies
\begin{eqnarray}\label{qestimate}
\|\mathbb D \varphi\|_2 +\|\mathbb D\tilde\varphi\|_2\lesssim |\Delta|^{\frac 1 2}.
\end{eqnarray}
The non-tangential maximal function operator $N_\ast$ acting on measurable functions $F$ on $\R^{n+2}_+$ was introduced in \eqref{eq:NTmaxDef-}. For $(x,t)\in \ree$ we also introduce the integrated non-tangential maximal function
\begin{eqnarray}
\label{eq:NTmaxDef}
\NT F(x,t)= \sup_{\lambda>0} \bigg(\bariiint_{\Lambda \times Q \times I} |F(\mu,y,s) |^2 \d \mu \d y \d s\bigg)^{1/2},
\end{eqnarray}
where $\Lambda=(\lambda/2, \lambda)$, $Q=B(x,\lambda)$ and $I= (t-\lambda^2, t+\lambda^2)$. If $g: \mathbb R^{n+1} \to\mathbb R $ and is locally integrable we let $\Max (g) $ be the $(n+1)$-dimensional (parabolic) Hardy-Littlewood maximal function
\begin{eqnarray*}
\Max (g) ( x,t) = \sup_{\varrho > 0} \bariint_{\Delta_\varrho ( x, t)} \, |g|  \d y \d s
\end{eqnarray*}
and we let $\Max_x$ and $\Max_t$ denote the standard (euclidean) Hardy-Littlewood maximal operators in the $x$ and $t$ variables only. Our construction of $F$ is then done through the following definition.

\begin{defn}
\label{setF} Let $\Delta$ be fixed and also fix $m = n+1$. Given $\kappa_0\gg1$, we let $F\subset 16\Delta$ be the set of all $(x,t)\in 16\Delta$ such that the following requirements are met:
\begin{eqnarray*}
  \mathrm{(i)}&&\Max(|\nabla_{x}  \varphi|^2)(x,t) + \Max(|\nabla_{x}  \tilde\varphi|^2)(x,t)\leq \kappa_0^2,\notag\\
  \mathrm{(ii)}&&\Max_x \Max_t(|\HT \dhalf\varphi|)(x,t) + \Max_x \Max_t(|\HT \dhalf\tilde\varphi|)(x,t)\leq \kappa_0,\notag\\
  \mathrm{(iii)}&&\mathbb D \varphi(x,t) + \mathbb D \tilde \varphi(x,t)\leq\kappa_0,\notag\\
  \mathrm{(iv)}&&N_\ast(\partial_\lambda P_{\lambda}^\ast\varphi)(x,t) + N_\ast(\partial_\lambda P_{\lambda}\tilde\varphi)(x,t)\leq \kappa_0,\notag\\
  \mathrm{(v)}&&\NT(\nabla_{x} P_{\lambda}^\ast\varphi)(x,t) + \NT (\nabla_{x} P_{\lambda}\tilde\varphi)(x,t)\leq \kappa_0.
\end{eqnarray*}
\end{defn}

Given $\Delta$ and $\kappa_0\gg1$, let $F$ be defined as above. Then, using the weak type $(1,1)$ of $\Max$, the strong type $(2,2)$ of $\Max_x \Max_t$, the estimates \eqref{ra1int} and \eqref{qestimate} and the $\L^2$-bounds for the non-tangential maximal functions that will later be obtained in Lemma~\ref{Est1+} and Lemma~\ref{Est1}, it follows that
\begin{eqnarray*}
  |16\Delta\setminus F|\lesssim (\kappa_0^{-2}+\kappa_0^{-1})|16 \Delta|.
\end{eqnarray*}
In particular,  we can now choose $\kappa_0$, depending only on $n$ and the ellipticity constants, so that
\begin{eqnarray}
  \label{ra3--}
  \frac{| 16\Delta\setminus F|}{ |16 \Delta|}\leq 1/1000.
\end{eqnarray}
This completes our construction of the set $F$ and from now on $\kappa_0$ is fixed as stated ensuring that \eqref{ra3--} holds.

\subsubsection*{Proof of the Carleson measure estimate}
Based on the previous steps, the proofs of Theorem~\ref{Ainfty+} and Theorem~\ref{Carleson for bounded} are reduced to verifying \eqref{Carl2}. To do this we construct, given $\Delta=\Delta_r$, $F\subset\Delta$ a Borel set and $\epsilon>0$, a parabolic sawtooth region above $F$ using parabolic cones of aperture $0<\eta\ll 1$. The parameter $\eta$ is an important degree of freedom in the argument. In \eqref{cutoff} we will construct a (smooth) cut-off function $\Psi=\Psi_{\eta,\epsilon}$ such that $\Psi(\lambda,x,t)=1$ on $F\times (2\epsilon,2r)$ and $\Psi(\lambda,x,t)=0$ if $\lambda\in (0,\epsilon)\cup (4r,\infty)$, and we let
	      \begin{eqnarray*}
	J_{\eta,\epsilon}:=\iiint_{\mathbb R^{n+2}_+}A\nabla_{\lambda,x} u\cdot\nabla_{\lambda,x} u\, \Psi^2\lambda \d x\d t\d\lambda.
	      \end{eqnarray*}
	  Then, by ellipticity of $A$,
	      \begin{eqnarray}\label{paris3uu-}
	\int_{2\epsilon}^{r}\iint_{F} |\partial_{\lambda } u |^2\,  \lambda\, \d x\d t\d\lambda\lesssim J_{\eta,\epsilon}.
	      \end{eqnarray}
Since $\Psi$ has compact support in the upper half space, we can ensure finiteness of $J_{\eta,\epsilon}$ and hence everything boils down to the following key lemma:

\begin{lem}[Key Lemma]\label{Carleson}
Let $\sigma, \eta\in (0,1)$ be given degrees of freedom. Then there exist a finite constant $c$ depending only on $n$ and the ellipticity constants, and a finite constant $\tilde c$ depending additionally on $\sigma$ and $\eta$, such that
		  \begin{eqnarray*}
		  J_{\eta,\epsilon}\leq (\sigma+c\eta)J_{\eta,\epsilon}+\tilde c|\Delta|.
		  \end{eqnarray*}
\end{lem}
Indeed, choosing $\sigma$ and $\eta$ small, both  depending at most on $n$ and the ellipticity constants, we first derive
	      \begin{eqnarray*}\label{paris3uu}
		  J_{\eta,\epsilon}\leq 2\tilde c|\Delta|,
	      \end{eqnarray*}
where now $\eta$ is fixed but $\tilde c$ is still independent of $\epsilon$. On letting $\epsilon\to 0$, we see from \eqref{paris3uu-} that the estimate \eqref{Carl2} holds. As discussed before, this completes the proofs of Theorem~\ref{Carleson for bounded} and Theorem~\ref{Ainfty+}.

\subsection{Organization of the paper} Section~\ref{sec2} is partly of preliminary nature and  we here prove \eqref{qestimate}. Section~\ref{sec3} is devoted to the important square function estimates underlying the proof of Theorem~\ref{Carleson}. These estimates rely on  recent results established in~\cite{AEN}. In Section~\ref{sec4} we prove the non-tangential maximal function estimates underlying the statements in Definition~\ref{setF}~(iv)-(v). Based on the material of Sections~\ref{sec2}-\ref{sec4} the set $F$ introduced in  Definition~\ref{setF} is well-defined  and we can ensure \eqref{ra3--}. In particular, thereby the set $F\subset 16\Delta$ is fixed as we proceed into  Section~\ref{sec5} and Section~\ref{sec6}. In Section~\ref{sec5} we then introduce sawtooth domains above $F$, we define the cut-off function $\Psi=\Psi_{\eta,\epsilon}$ referred to above and we prove some auxiliary Carleson measure estimates. The proof of Lemma~\ref{Carleson} is given in Section~\ref{sec6}.

\section{Technical tools}\label{sec2}

In this section we collect three technical lemmas that shall prove useful in the further course.
We begin with standard Caccioppoli estimate which we here state without proof.

\begin{lem}[Caccioppoli estimate] \label{lem:Caccioppoli}
Let $u$ be a weak solution to $\pd_{t} u - \div_{\lambda,x} A \gradlamx u+\alpha u=0$ on $\mathbb R^{n+2}_+$ where $\alpha \in \L^\infty(\mathbb R^{n+2}_+)$, $\alpha \ge 0$,
and let $\psi\in \C_0^\infty(\mathbb R^{n+2}_+)$. Then
\begin{align*}
\iiint |\gradlamx u|^2\psi^2 \d x \d\lambda \d t \leq c  \iiint |u|^2\bigl(|\gradlamx \psi|^2+|\psi||\pd_{t}\psi|\bigr) \d x \d\lambda \d t
\end{align*}
for some finite constant $c$ depending on $n$ and the ellipticity constants of $A$.
\end{lem}

Next, we record a Poincar\'{e}-type estimate for functions in the homogeneous energy space $\dot{\E}(\R^{n+1})$. We use the standard notation for parabolic cubes introduced in Section~\ref{parabolic measure}.

\begin{lem}
\label{Poincare}
Let $v \in \dot{\E}(\R^{n+1})$ and let $\Delta_\varrho = \Delta_\varrho(x_0,t_0) \subset\mathbb R^{n+1}$ be a parabolic cube. Then
\begin{eqnarray*}
\frac{1}{\varrho} \bariint_{\Delta_\varrho} \biggl | v-\bariint_{\Delta_\varrho}v \biggr |\, \d x\d t \lesssim \Max (|\nabla_{x}v|)(x_0,t_0)+ \Max_x \Max_t (|\HT \dhalf v|)(x_0, t_0).
\end{eqnarray*}
\end{lem}

\begin{proof} We write  $\Delta_\varrho= Q_\varrho\times I_\varrho$ and we let
$$f(t):=\barint_{Q_\varrho}v(x,t)\, \d x,$$ noting that this function is contained in the homogeneous fractional Sobolev space $\Hdot^{1/2}(\R)$, see Section~3.1 in \cite{AEN}.
Then
\begin{eqnarray*}
\bariint_{\Delta_\varrho}  \biggl |v-\bariint_{\Delta_\varrho}v\biggr |\, \d x\d t\lesssim \varrho \Max(|\nabla_{x}v|)(x_0,t_0)
+ \barint_{I_\varrho} \bigg|f - \barint_{I_\varrho} f \bigg| \, \d t
\end{eqnarray*}
by Poincar\'{e}'s inequality in the spatial variable $x$ only. Furthermore, for $f \in \Hdot^{1/2}(\R)$ we have at hand the non-local Poincar\'{e} inequality
\begin{eqnarray*}
\barint_{I_\varrho} \bigg|f - \barint_{I_\varrho} f \bigg| \, \d t \leq \varrho \sum_{k \in \IZ} \frac{1}{1+|k|^{3/2}} \barint_{k \varrho^2 + I_\varrho} |\HT \dhalf f| \, \d t,
\end{eqnarray*}
see Lemma~8.3 in~\cite{AEN}. Rearranging the covering of the real line by translates of $I_\varrho$ into a covering by dyadic annuli, we obtain
\begin{eqnarray*}
\barint_{I_\varrho} \bigg|f - \barint_{I_\varrho} f \bigg| \, \d t
&\leq& \varrho\sum_{m\geq 0}2^{-m}\barint_{4^m{I_\varrho}} |\HT \dhalf f| \, \d \sigma\notag\\
&\leq & \varrho\sum_{m\geq 0}2^{-m}\bariint_{Q\times 4^m{I_\varrho}} |\HT \dhalf v|\,  \d x\d t\notag\\
&\leq& 2 \varrho \Max_x (\Max_t (|\HT \dhalf v|)(x_0,t_0),
\end{eqnarray*}
where the second step can rigorously be justified using Fubini's theorem, see Lemma~3.10 in~\cite{AEN}.
\end{proof}

As a consequence, we obtain an important estimate for the parabolic maximal differential operator $\mathbb D$ defined in \eqref{lipll+}.

\begin{lem}
\label{Quotient}
The operator $\mathbb D$ maps $\dot{\E}(\R^{n+1})$ boundedly into $\L^2(\ree)$.
\end{lem}

\begin{proof} Let $v \in \dot{\E}(\R^{n+1})$.  We first claim that
\begin{equation}\label{liplu}
\begin{split}
\frac{|v(x,t)-v(y,s)|}{\|(x-y,t-s)\|}&\lesssim\Max (|\nabla_{x}v|)(x,t)+ \Max_x \Max_t (|\HT \dhalf v|)(x,t) \\
&\quad+\Max(|\nabla_{x}v|)(y,s)+ \Max_x \Max_t (|\HT \dhalf v|)(y,s)
\end{split}
\end{equation}
holds for almost every $(x,t)$, $(y,s)\in\mathbb R^{n+1}$. Indeed, let $(x,t)$ be a Lebesgue point for $v$ and for $\varrho > 0$ let $v_\varrho$ denote the average of $v$ over the parabolic cube $\Delta_\varrho:=\Delta_\varrho(x,t)$. Then, by a telescoping sum and an application of Lemma~\ref{Poincare},
\begin{eqnarray*}
|v(x,t)-v_\varrho|&\leq & \sum_{k=0}^\infty |v_{2^{-k-1}\varrho}-v_{2^{-k}\varrho}| \notag\\
&\lesssim & \sum_{k=0}^\infty 2^{-k}\varrho \biggl (\Max (|\nabla_{x}v|)(x,t)+ \Max_x \Max_t (|\HT \dhalf v|)(x,t)\biggr ) \notag\\
&\leq &  2 \varrho \biggl (\Max (|\nabla_{x}v|)(x,t)+ \Max_x \Max_t (|\HT \dhalf v|)(x,t)\biggr ).
\end{eqnarray*}
Furthermore, let also $(y,s)$ be a Lebesgue point for $v$ and assume that $(y,s)\in \Delta_\varrho(x,t)$. Then $\Delta_\varrho(x,t)\subset \Delta_{2\varrho}(y,s)$ and we obtain as above,
\begin{eqnarray*}
|v(y,s)-v_\varrho|
&\leq&\biggl|v(y,s)-\bariint_{\Delta_{2\varrho}(y,s)}v\biggr|+\biggl|\bariint_{\Delta_{2\varrho}(y,s)}v-v_\varrho \biggr|\notag\\
&\lesssim&\varrho \biggl (\Max (|\nabla_{x}v|)(y,s)+ \Max_x \Max_t (|\HT \dhalf v|)(y,s)\biggr).
\end{eqnarray*}
Now, for $(x,t)\neq (y,s)$ as above we can specify $\varrho := ||(x-y,t-s)||$ and \eqref{liplu} follows by adding up the previous two estimates.
In particular, we obtain
\begin{eqnarray*}
\mathbb D v(x,t)
&\lesssim & \Max (|\nabla_{x}v|)(x,t)+ \Max_x \Max_t (|\HT \dhalf v|)(x,t) \\
&& + \Max \Max (|\nabla_{x}v|)(x,t)+
\Max \Max_x \Max_t (|\HT \dhalf\varphi|)(x,t)
\end{eqnarray*}
for almost every $(x,t)\in\mathbb R^{n+1}$ and since all occurring maximal operators are $\L^2$-bounded, we conclude $\|\mathbb D v\|_2 \lesssim \|\nabla_x v\|_2 + \|\HT \dhalf v\|_2$ as required.
\end{proof}

\section{Functional calculus and square function estimates} \label{sec3} In this section we prove the important square function estimates for $\cH_\pa$ and $\cH_{\pa}^\ast$ underlying the proof of Lemma~\ref{Carleson}. Most of this material is taken from~\cite{AEN}.

Given $\mu \in (0, \pi/2)$ we let
\begin{align*}
\S_\mu := \{z \in \IC: |\arg z| < \mu \text{ or } |\arg z - \pi| < \mu \}
\end{align*}
denote the open double sector of angle $\mu$. We let
$$\Psi(\S_\mu):=\{\psi\in H^\infty(\S_\mu): \exists\  \alpha>0,\ C>0,\mbox{ such that }|\psi(z)|\leq C\min\{|z|^\alpha,|z|^{-\alpha}\}\}$$
where $H^\infty(\S_\mu)$ is the set of all bounded holomorphic functions on $\S_\mu$. Furthermore, recall that an operator $T$ in a Hilbert space is \emph{bisectorial} of \emph{angle} $\omega \in (0, \pi/2)$ if its spectrum is contained in the closure of $ \S_\omega $ and if, for each $\mu \in (\omega, \pi/2)$, the map $z \mapsto z(z - T)^{-1}$ is uniformly bounded on $\IC \setminus \S_\mu$. In this case a bounded operator $\psi(T)$ is defined by the functional calculus for bisectorial operators and we refer to~\cite{Mc} or~\cite{EigeneDiss} for the few essentials of this theory used in this section.
Turning to concrete operators, we represent vectors $h \in \mathbb C^{n+2}$ as
\begin{equation*}
h= \begin{bmatrix} h_\pe \\ h_\pa \\ h_\te \end{bmatrix},
\end{equation*}
where the normal part $h_\pe$ is scalar valued, the tangential part $h_{\pa}$ is valued in $\mathbb C^n$ and the time part $h_{\te}$ is again scalar valued and let
\begin{equation*}
\P:= \Pfull, \qquad \M:= \Mfull.
\end{equation*}
Here, $\M$ is considered as a bounded multiplication operator on $\L^2(\mathbb R^{n+1};\mathbb C^{n+2})$ and the parabolic Dirac operator $\P$ is an unbounded operator in $\L^2(\mathbb R^{n+1};\mathbb C^{n+2})$ with maximal domain. The link with the parabolic operator $\cH_\pa$ is that $(\P \M)^2$ and $(\M \P)^2$ are operator matrices in block form
\begin{equation}
\label{eq:PM square}
(\P \M)^2 = \begin{bmatrix} \cH_\pa & 0 & 0 \\ 0& \ast & \ast \\ 0 & \ast & \ast \end{bmatrix}, \qquad (\M \P)^2 = \begin{bmatrix} \cH_\pa & 0 & 0 \\ 0& \ast & \ast \\ 0 & \ast & \ast \end{bmatrix},
\end{equation}
where the entries $\ast$ do not play any role in the following but of course they could be computed explicitly. Note that taking adjoints    in \eqref{eq:PM square}, hence using $(\P^\ast\M^\ast)^2$ or $(\M^\ast\P^\ast)^2$, allows to obtain $\cH_\pa^\ast$.
The following theorem provides square function estimates.

\begin{thm}
\label{thm:bhfc} The operator $\P\M$ is a bisectorial operator in $\L^2(\mathbb R^{n+1};\mathbb C^{n+2})$ with angle $\omega$ of bisectoriality depending only upon $n$ and the ellipticity constants of $A$ and the same range as $\P$, that is, $\ran(\P\M) = \ran(\P)$. Let  $\mu \in (\omega, \pi/2)$ and consider $\psi\in \Psi(\S_\mu)$ non vanishing on each connected component of $\S_\mu$. Then
\begin{align*}
  \int_0^\infty\| \psi(\lambda PM) h \|_{2}^2 \, \frac{\mathrm{d}\lambda{}}\lambda{} \sim \|h\|_{2}^2 \mbox{\quad if $h \in \cl{\ran(\P \M)}$}
\end{align*}
and the implicit constants in this estimate depend only upon $n$, the ellipticity constants of $A$, $\mu$ and $\psi$.
The same holds true for $\M \P$ on $\cl{\ran(\M \P)} = \M \cl{\ran(\P)}$ and with $\P\M$, $\M \P$, replaced by $\P^\ast\M^\ast$, $\M^\ast \P^\ast$.
\end{thm}
\begin{proof} For $\P\M$, this is a mere consequence of Theorem 2.3 in~\cite{AEN}: Indeed, this theorem states all assertions apart from that only the quadratic estimate
\begin{align*}
  \int_0^\infty\| {\lambda{}}\P \M(\id+{\lambda{}^2}\P\M\P\M)^{-1} h \|_{2}^2 \, \frac{\mathrm{d}\lambda{}}\lambda{} \sim \|h\|_{2}^2 \mbox{\qquad for $h \in \cl{\ran(\P \M)}$}
\end{align*}
is mentioned.
But due to a general result on quadratic estimates for bisectorial operators on Hilbert spaces, see~\cite{Mc} or Theorem~3.4.11 in~\cite{EigeneDiss}, this quadratic estimate is in fact equivalent to the set of quadratic estimates stated above. The statement for $\M\P$ follows from the fact that  this operator is similar to $\P\M$ on their respective ranges by $\M\P= \M(\P\M)\M^{-1}$. The statements for $\P^\ast\M^\ast$, $\M^\ast \P^\ast$ follow by duality, see again \cite{Mc, EigeneDiss}.
\end{proof}

Below, we single out some particular instances of the theorem above and reformulate them in terms of $\cH_\pa$ and $\cH_\pa^*$ to have direct references later on. Throughout, we let $\varphi$, $\tilde\varphi$ be as in \eqref{eq7+int}, \eqref{ra1int} and we recall that the resolvent operators $\P_{\lambda}^\ast$, $\P_{\lambda}$ were defined in \eqref{paris3} for the moment with $m$ unspecified.

\begin{lem}\label{square est}
There exists $c$, $1\leq c<\infty$, depending only on $n$, the ellipticity constants and $m\ge 1$ such that
\begin{eqnarray*}
\mathrm{(i)}&&\iiint_{\mathbb R^{n+2}_+} |\partial_\lambda\P_{\lambda}^\ast\varphi|^2+|\partial_\lambda\P_{\lambda}\tilde \varphi|^2 \, \frac {\d x\d t\d\lambda}{\lambda}\leq c|\Delta|,\notag\\
\mathrm{(ii)}&&\iiint_{\mathbb R^{n+2}_+} |\lambda\nabla_{x}\partial_\lambda\P_{\lambda}^\ast\varphi|^2+|\lambda\nabla_{x}\partial_\lambda\P_{\lambda}\tilde \varphi|^2 \, \frac {\d x\d t\d\lambda}{\lambda}\leq c|\Delta|,\notag\\
\mathrm{(iii)}&&\iiint_{\mathbb R^{n+2}_+} |\lambda\mathcal{H}_{\pa}^\ast\P_{\lambda}^\ast\varphi|^2+|\lambda\mathcal{H}_{\pa}\P_{\lambda}\tilde \varphi|^2 \, \frac {\d x\d t\d\lambda}{\lambda}\leq c|\Delta|, \\
\mathrm{(iv)}&&\iiint_{\mathbb R^{n+2}_+} |\lambda^2\mathcal{H}_{\pa}^\ast \partial_\lambda \P_{\lambda}^\ast\varphi|^2+|\lambda^2\mathcal{H}_{\pa}\partial_{\lambda} \P_{\lambda}\tilde \varphi|^2 \, \frac {\d x\d t\d\lambda}{\lambda}\leq c|\Delta|.
\end{eqnarray*}

\end{lem}
\begin{proof} In the following we will only prove the estimates for $\P_{\lambda}\tilde \varphi$, the estimates for $\P_{\lambda}^\ast\varphi$ being proved similarly with $\P^\ast $ and $\M^\ast$ replacing $\P$ and  $\M$. Note that $\tilde \varphi  \in \dot{\E}(\R^{n+1})$ and hence the following calculations can be justified, for example, by approximating $\tilde \varphi$ by smooth and compactly supported functions in the semi-norm of $\dot{\E}(\R^{n+1})$. Keeping this in mind, we may directly argue with $\tilde \varphi$. We begin with (iii). Let
\begin{align*} h:= \begin{bmatrix}
0 \\
  -\nabla_{x}\tilde\varphi  \\ - \HT\dhalf\tilde\varphi
\end{bmatrix}=\P\begin{bmatrix}
\tilde\varphi \\
  0  \\ 0
\end{bmatrix} \in \cl{\ran(\P)}=\cl{\ran(\P \M)},
\end{align*}
and note, using \eqref{eq:PM square} and elementary manipulations of resolvents of $\P \M$ and $\M \P$, that
\begin{eqnarray*}
\begin{bmatrix} \lambda \cH_\pa P_\lambda \tilde \varphi \\ 0 \\ 0 \end{bmatrix}
= \lambda (\M \P)^2 (1 + (\lambda \M \P)^2)^{-m} \begin{bmatrix} \tilde \varphi \\ 0 \\ 0 \end{bmatrix}
= \M (\lambda \P \M) (1 + (\lambda \P \M)^2)^{-m} P \begin{bmatrix} \tilde \varphi \\ 0 \\ 0 \end{bmatrix} = \M \psi(\lambda \P \M) h
\end{eqnarray*}
where $\psi(z):=z(1+z^2)^{-m}$. Hence,
\begin{eqnarray*}
\iiint_{\mathbb R^{n+2}_+}|\lambda\mathcal{H}_{\pa}\P_{\lambda}\tilde \varphi|^2\, \frac {\d x\d t\d\lambda}{\lambda}
\lesssim \|h\|_{2}^2=||\nabla_x\tilde\varphi||_2^2+||H_tD_{t}^{1/2}\tilde\varphi||_2^2\lesssim|\Delta|,
\end{eqnarray*}
by an application of Theorem~\ref{thm:bhfc} and \eqref{ra1int}. This proves (iii). Likewise, (i) and (iv) follow with  $\psi(z)=-2mz^2(1+z^2)^{-m-1}$ and $\psi(z)=-2mz^3(1+z^2)^{-m-1}$, respectively.
Finally, to prove (ii) we write analogously
\begin{align*}
\begin{bmatrix} 0   \\ -\lambda \nabla_x \partial_\lambda \P_{\lambda}\tilde\varphi \\ -\lambda \HT \dhalf \partial_\lambda \P_{\lambda}\tilde\varphi\end{bmatrix}
= \P \begin{bmatrix} \lambda \partial_\lambda P_\lambda \tilde \varphi \\ 0 \\ 0 \end{bmatrix}
= -2m \P(\lambda \M \P)^2 (1+ (\lambda \M \P)^2)^{-m-1} \begin{bmatrix} \tilde \varphi \\ 0 \\ 0 \end{bmatrix}
= \tilde \psi(\lambda \P \M) h
\end{align*}
with $\tilde \psi(z) = -2m z^2 (1+z^2)^{-m-1}$ and the claim follows by
yet another application of Theorem~\ref{thm:bhfc}. \end{proof}

\begin{lem}\label{square1}
There exists $c$, $1\leq c<\infty$, depending only on $n$, the ellipticity constants and $m\ge 1$ such that
\begin{eqnarray*}
			  \iiint_{\mathbb R^{n+2}_+} |(I-\P_{\lambda}^\ast)\varphi|^2+|(I-\P_{\lambda})\tilde\varphi|^2 \, \frac {\d x\d t\d\lambda}{\lambda^3}\leq c|\Delta|.
			    \end{eqnarray*}
\end{lem}
\begin{proof}  We have $$(I-\P_{\lambda})\tilde\varphi=\int_0^{\lambda}\partial_\sigma\P_\sigma \tilde\varphi\, \d\sigma.$$ Applying Hardy's inequality and Lemma~\ref{square est}~(i) we see that
\begin{eqnarray*}
			  \iiint_{\mathbb R^{n+2}_+}|(I-\P_{\lambda})\tilde\varphi|^2\, \frac {\d x\d t\d\lambda}{\lambda^3}\lesssim
			  \iiint_{\mathbb R^{n+2}_+}|\partial_\lambda \P_{\lambda}\tilde\varphi|^2\, \frac {\d x\d t\d\lambda}{\lambda}\leq c|\Delta|.
			    \end{eqnarray*}
			    The proof of the estimate for $(I-\P_{\lambda}^\ast)\varphi$ is similar.
\end{proof}

\section{Non-tangential maximal function estimates}\label{sec4}
The pointwise non-tangential maximal operator $N_\ast$ was introduced in \eqref{eq:NTmaxDef-} and its integrated version $\NT$ was defined in \eqref{eq:NTmaxDef}. In this section we use the previously obtained square function estimates to derive bounds for these maximal functions.

\begin{thm}
\label{thm:NTmax} Let $ h\in \cl{\ran(\P \M)}$ and let $F(\lambda,x,t) = (\e^{-\lambda [\P \M]}h)(x,t)$, with $[\P \M] := \sqrt{ (PM)^2}$. Then
\begin{equation*}
\| \NT F \|_{2} \sim \|h\|_{2},
\end{equation*}
where the implicit constants depend only on dimension and the ellipticity constants of $A$.
The conclusion remains true also with $\P \M$ replaced by $\P^\ast \M^\ast$.
\end{thm}

\begin{proof} For $\P\M$, this is Theorem~2.12 in~\cite{AEN}. The same statement can be proved for $\P^\ast\M^\ast$.
\end{proof}

In the following $\P_{\lambda}^\ast\varphi$, $\P_{\lambda}\tilde \varphi$ are again as defined in \eqref{paris3}.

\begin{lem}
\label{Est1+}
There exists $c$, $1\leq c<\infty$, depending only on $n$, the ellipticity constants and $m\ge 1$ such that
\begin{eqnarray*}
\|\NT(\nabla_{x} P_{\lambda}^\ast\varphi)\|_2^2+\|\NT (\nabla_{x} P_{\lambda}\tilde\varphi)\|_2^2\leq c|\Delta|.
\end{eqnarray*}
\end{lem}
\begin{proof} We only give the proof of the estimate of $\NT (\nabla_{x} P_{\lambda}\tilde\varphi)$. To start the proof we first note as in the proof of Lemma~\ref{square est}~(ii) that
\begin{align*}
\begin{bmatrix} 0   \\ -\nabla_x \P_{\lambda}\tilde\varphi \\ - \HT D_{t}^{1/2}\P_{\lambda}\tilde\varphi\end{bmatrix}= \vartheta(\lambda \P\M)h, \qquad h = \begin{bmatrix} 0 \\ - \gradx \tilde \varphi \\ - \HT \dhalf \tilde \varphi \end{bmatrix} = P \begin{bmatrix} \tilde \varphi \\ 0 \\ 0 \end{bmatrix} \in \cl{\ran(\P \M)}
\end{align*}
where now $\vartheta (z) = (1+z^2)^{-m}$. Thus, $-\nabla_x \P_{\lambda}\tilde\varphi=(\vartheta (\lambda\P \M)h)_{\pa}$ and we have to estimate $\|\NT(\vartheta (\lambda \P \M)h)\|_2$. To this end, we first note
\begin{eqnarray}
\label{Est1++}
\|\NT(\e^{-\lambda [\P \M]}h) \|_2^2 \lesssim \|h\|_2^2 = \|\gradx  \tilde\varphi \|_2^2 + \|\HT \dhalf\tilde\varphi\|_2^2 \lesssim |\Delta|,
\end{eqnarray}
using Theorem~\ref{thm:NTmax}, the construction of $h$ and \eqref{ra1int}. Now let $\psi(z):= \vartheta (z) - e^{-\sqrt{z^2}}$.
Tonelli's theorem yields
			    \begin{eqnarray*}
\|\NT(\psi(\lambda \P \M)h)\|_2^2 \lesssim \int_0^\infty \|\psi(\lambda \P \M)h\|_2^2 \, \frac{\d \lambda}{\lambda},
\end{eqnarray*}
see for example Lemma~8.10 in~\cite{AEN} for an explicit proof. Since $\psi \in \Psi(\S_\mu)$ for every $\mu \in (0, \pi/2)$, we deduce from Theorem~\ref{thm:bhfc} that
\begin{eqnarray*}
\|\NT(\psi(\lambda \P \M)h)\|_2^2 \lesssim \|h\|_2^2 \lesssim |\Delta|,
\end{eqnarray*}
which in combination with \eqref{Est1++} yields the claim.
\end{proof}

For the $\lambda$-derivatives of $P_\lambda^\ast \varphi$ and $P_\lambda \tilde \varphi$ we could get $\L^2$-bounds for the integrated non-tangential maximal function immediately from the square function estimate in Lemma~\ref{square est}~(i). However, this would not be enough for our purpose. To derive the required bounds for the \emph{pointwise} non-tangential maximal function, we need the following lemma.

\begin{lem} \label{resolvent kernel}
For $\lambda > 0$ and $m \ge 1$, the resolvent $P_{\lambda}=(1+\lambda^2 \cH_\pa)^{-m}$, defined as a bounded operator on $\L^2(\ree)$, is represented by an integral kernel $K_{\lambda,m}$ with pointwise bounds
\begin{equation}
\label{eq:gaussian}
  |K_{\lambda,m} (x,t,y,s)| \leq \frac{C 1_{(0,\infty)}(t-s)}{\lambda^{2m}} (t-s)^{-n/2+m-1}  \e^{-\frac{t - s}{\lambda^2}} \e^{-c \frac{|x-y|^2}{t-s}},
\end{equation}
where $C,c>0$ depend only on $n$, the ellipticity constants and $m$. An analogous representation holds for $(1+\lambda^2 \cH_\pa^*)^{-m}$  with adjoint kernel $K^\ast_{\lambda,m}$.
\end{lem}

\begin{proof} It suffices to do it when $m=1$ as iterated convolution in $(x,t)$ of the estimate on the right hand side of \eqref{eq:gaussian} with $m=1$ yields the result.

Let $f \in \C^\infty_{0}(\ree)$. Let $u=(1+\lambda^2\cH_\pa)^{-1}f$ given by the functional calculus of $\cH_\pa$. Then $u\in \L^2(\ree)$ and, in particular, $u$ is a weak solution to $\lambda^2\partial_{t}u-\lambda^2 \divx A_{\pa\pa}\gradx u + u= f$. On the other hand, by Aronson's result~\cite{A}, the operator $\cH_\pa$ has a fundamental solution, denoted by $K(x,t,y,s)$, having bounds
\begin{eqnarray*}
|K(x,t,y,s)| \leq {C} 1_{(0,\infty)}(t-s) \cdot (t-s)^{-n/2}  \e^{-c \frac{|x-y|^2}{t-s}}  \qquad \mbox{for $x,y \in \R^n$, $t,s \in \R$}
\end{eqnarray*}
with constants $C,c$ depending only on dimension and the ellipticity constants, and satisfying
\begin{eqnarray}
\label{conservation}
\int_{\R^n} K(x,t,y,s) \, \d y = 1 \qquad \mbox{for $x \in \R^n$, $t,s \in \R$, $t>s.$}
\end{eqnarray}
Set $K_{\lambda,1}(x,t,y,s)=  \lambda^{-2}K(x,t,y,s) \e^{- \frac{t - s}{\lambda^2}}$ and $v(x,t)= \iint_\ree K_{\lambda,1}(x,t,y,s) f(y,s) \d y \d s$.  Aronson's estimate implies $v\in \L^2(\ree)$ and a calculation shows that
$v$ is a weak solution to the same equation as $u$. Thus, $w:=u-v$ is a weak solution of
$\partial_{t}w-\divx A_{\pa\pa}\gradx w + \lambda^{-2} w= 0$ and we may use the Caccioppoli estimate of Lemma~\ref{lem:Caccioppoli} in $\ree$. Choosing test functions $\psi$ that converge to $1$ reveals $\gradx w=0$ as $w\in \L^2(\ree)$. Hence $w$ depends only on $t$. Again, as $w\in \L^2(\ree)$, $w$ must be $0$. This shows that $P_{\lambda}f$ has the desired representation for all $f \in \C^\infty_{0}(\ree)$ and we conclude by density. \end{proof}

\begin{rem} \label{kernel on dotE}
The kernel representation from Lemma~\ref{resolvent kernel} can be extended from $\L^2(\ree)$ to $\dot{\E}(\R^{n+1})$ since the latter embeds continuously into $\L^2(\ree) + \L^\infty(\ree)$ modulo constants, see for example Lemma~3.11 in~\cite{AEN}. In this sense $(1+\lambda^2 \cH_\pa)^{-m} 1 = 1$ holds due to \eqref{conservation}.
\end{rem}

\begin{lem} \label{Est1} Fix $m=n+1$ in the definitions of $P_\lambda^\ast$ and $P_{\lambda}$.
There exists $c$, $1\leq c<\infty$, depending only on $n$ and the ellipticity constants such that
\begin{eqnarray*}
||N_\ast(\partial_\lambda P_{\lambda}^\ast\varphi)||_2^2+||N_\ast(\partial_\lambda P_{\lambda}\tilde\varphi)||_2^2\leq c | \Delta|.
\end{eqnarray*}
\end{lem}
\begin{proof} By symmetry of definitions, we only have to prove one of the estimate  and we do the one of $N_\ast(\partial_\lambda P_{\lambda}^\ast\varphi)$ for a change.

To start the proof, fix $(\mu,y,s)\in W(\lambda,x,t)$, where $W(\lambda,x,t):=\Lambda_\lambda \times Q_\lambda(x) \times I_\lambda(t)= $ and $\Lambda_\lambda=(\lambda/2, \lambda)$, $Q_\lambda(x)=B(x,\lambda)$ and $I_\lambda(t)= (t-\lambda^2, t+\lambda^2)$ is one of the Whitney regions used in the definition of $N_\ast$ and recall that $\Delta_{\lambda}(x,t)= Q_\lambda(x) \times I_\lambda(t)$. Let $\sigma\in \Lambda_\lambda$ be arbitrary for the moment. We note that within the functional calculus for $\cH_\pa^\ast$,
				      \begin{eqnarray*}
			    \partial_\mu P_{\mu}^\ast=-2m\mu
					\mathcal{H}_{\pa}^\ast(1+\mu^2\mathcal{H}_{\pa}^\ast)^{-m-1},
					\end{eqnarray*}
					and we introduce $\tilde P_{\mu}^\ast:=(1+\mu^2\mathcal{H}_{\pa}^\ast)^{-1}$ to write
					  \begin{eqnarray*}
			    \partial_\mu P_{\mu}^\ast \varphi
					= -\frac {2m \mu}\sigma (1+\mu^2\mathcal{H}_{\pa}^\ast)^{-m}(1+\mu^2\mathcal{H}_{\pa}^\ast)^{-1} {(1+\sigma^2\mathcal{H}_{\pa}^\ast)} \sigma \cH_\pa^\ast \tilde P_{\sigma}^\ast \varphi.
					  \end{eqnarray*}
					  It is convenient to expand this identity as
					  \begin{eqnarray}\label{reproducing}
			      \partial_\mu P_{\mu}^\ast \varphi =  -\frac {2m \mu}\sigma (1+\mu^2\mathcal{H}_{\pa}^\ast)^{-m} \left( \frac{\sigma^2}{\mu^2} + \biggl(1 - \frac{\sigma^2}{\mu^2} \biggr) (1+ \mu^2 \cH_\pa^\ast)^{-1} \right) \sigma \cH_\pa^\ast \tilde P_{\sigma}^\ast \varphi
					  \end{eqnarray}
					  since this reveals $\partial_\mu P_{\mu}^\ast \varphi = T (\sigma \cH_\pa^\ast \tilde P_{\sigma}^\ast \varphi)$, where the operator $T$ is given by a linear combination of the resolvent kernels $K^\ast_{\mu,m}$ and $K^\ast_{\mu, m+1}$ provided by Lemma~\ref{resolvent kernel}. 							  Setting $G_0(x,t):=\Delta_{2 \lambda}(x,t)$ and $G_j(x,t):=\Delta_{2^{j+1} \lambda}(x,t) \setminus \Delta_{2^j \lambda}(x,t)$, $j \geq 1$,	                since $(\mu, y,s) \in \Lambda_\lambda \times \Delta_\lambda(x,t)$, we can infer pointwise estimates
					  \begin{eqnarray*}
			      |K_{\mu,m+k}^\ast(y,s,z,\tau)| \leq \frac{C}{\lambda^{n+2}} \e^{-c4^j} \mbox{\qquad if $(z,\tau) \in G_j(x,t),\ j \geq 0,\ m+k \geq n/2+1$},
					  \end{eqnarray*}
					  where $C,c>0$ depend only on $n$, the ellipticity constants and $m+k$. Note  that the bound for $j = 0$ only holds since $m+k \geq m  = n+1\ge  n/2+1$ guarantees that $K_{\mu,m+k}^\ast$ is \emph{bounded}. As we have $\lambda/2 < \sigma < \lambda$, the kernel $K^\ast$ of the operator acting on $\sigma \cH_\pa^\ast \tilde P_{\sigma}^\ast \varphi$ on the right-hand side of \eqref{reproducing} has analogous bounds and we can eventually record
					  \begin{eqnarray*}
			      |\partial_\mu P_{\mu}^\ast \varphi(y,s)|
					  &=& \bigg| \iint_\ree K^\ast(y,s,z,\tau) \sigma \cH_\pa^\ast \tilde P_{\sigma}^\ast \varphi(z,\tau) \, \d z \d \tau \bigg| \\					
					  &\leq& \sum_{j=0}^\infty C 2^{j(n+2)} \e^{-c4^j} \bariint_{G_j(x,t)} |\sigma \cH_\pa^\ast \tilde P_{\sigma}^\ast \varphi(z,\tau)| \, \d z \d \tau
					  \end{eqnarray*}
					  with $C,c>0$ depending only on $n$ and the ellipticity constants. As $(\mu,y,s) \in W(\lambda,x,t)$ was arbitrary in this argument, we have in fact
					  \begin{eqnarray}\label{reproducing bound}
					  \sup_{(\mu,y,s)\in W(\lambda,x,t)} |\partial_\mu P_{\mu}^\ast \varphi(y,s)|^2 \lesssim\sum_{j=0}^\infty e^{-c4^j} \bariint_{2^{j+1}Q_\lambda(x) \times 4^{j+2} I_\lambda(t)} |\sigma \cH_\pa^\ast \tilde P_{\sigma}^\ast \varphi(z,\tau)|^2 \, \d z \d \tau,
					  \end{eqnarray}
					  where we have also used Cauchy-Schwarz to switch to $\L^2$-averages and exploited the exponential decay. Since only the right-hand side depends on  $\sigma \in \Lambda_\lambda$, we can average in $\sigma$ and take the supremum in $\lambda$ to find
					  \begin{eqnarray*}
					    N_\ast(\partial_\lambda P_\lambda^\ast \varphi)(x,t)^2 \lesssim\sum_{j=0}^\infty e^{-c4^j} \sup_{\lambda > 0} \int_{\lambda/2}^\lambda \bariint_{2^{j+1}Q_\lambda(x) \times 4^{j+2} I_\lambda(t)} |\sigma \cH_\pa^\ast \tilde P_{\sigma}^\ast \varphi(z,\tau)|^2 \, \frac{\d z \d \tau \d \sigma}{\sigma}.
					  \end{eqnarray*}
					  By a direct application of Tonelli's theorem, see Lemma~8.10 in~\cite{AEN} for an explicit proof, this implies
					  \begin{eqnarray*}
					  \iint_\ree |N_\ast(\partial_\lambda P_\lambda^\ast \varphi)(x,t)|^2 \, \d x \d t
					  \lesssim \sum_{j = 0}^\infty \e^{-c 4^j} \iiint_{\R^{n+2}_+} |\sigma \cH_\pa^\ast \tilde P_{\sigma}^\ast \varphi(z,\tau)|^2 \, \frac{\d z \d \tau \d \sigma}{\sigma}
					  \end{eqnarray*}
					  and hence the claim follows from Lemma~\ref{square est}~(i) applied with $m = 1$.
\end{proof}

\section{Parabolic sawtooth domains associated with $F$}\label{sec5}

Throughout this section, let $\Delta$ and $\kappa_0\gg1$ be given and let $F\subset 16\Delta $ be the set introduced in Definition~\ref{setF} with $ P_{\lambda}=(1+\lambda^2\mathcal{H}_{\pa})^{-n-1}$,  $ P_{\lambda}^\ast=(1+\lambda^2\mathcal{H}_{\pa}^\ast)^{-n-1}$ from now on. Let us recall that the non-tangential maximal operators $N_\ast$ and $\NT$ at $(x_0,t_0) \in \R^{n+1}$ are defined with reference to the Whitney regions $\Lambda_\lambda \times Q_\lambda(x_0) \times I_\lambda(x_0)$, where $\Lambda_\lambda=(\lambda/2, \lambda)$, $Q_\lambda(x_0)=B(x_0, \lambda)$, $I_\lambda(t_0)= (t_0-\lambda^2, t_0+\lambda^2)$, and that $\Gamma(x_0,t_0)$ denotes the parabolic cone with vertex $(x_0,t_0)$ and aperture one, see \eqref{cone}. In particular, we have
    \begin{eqnarray*}
\Gamma(x_0,t_0)\subset \bigcup_{\lambda>0}\Lambda_\lambda\times Q_\lambda(x_0)\times I_\lambda(t_0).
\end{eqnarray*}
Next, we introduce a sawtooth domain associated with $F$,
\begin{eqnarray*} 
  \Omega:=\bigcup_{(x_0,t_0)\in F\cap \Delta}\Gamma(x_0,t_0),
\end{eqnarray*}
and establish pointwise estimates for the differences
\begin{eqnarray*}
  \theta_\lambda:=\varphi-\P_{\lambda}^\ast\varphi, \ \tilde\theta_\lambda:=\tilde\varphi-\P_{\lambda}\tilde\varphi.
\end{eqnarray*}

	  \begin{lem}\label{saw1}
		    We have
				\begin{eqnarray*}\label{ra3ka}
				\mathrm{(i)}&& |\theta_\lambda(x,t)|+|\tilde\theta_\lambda(x,t)|\leq \kappa_0 \lambda\mbox{\qquad if $(\lambda, x,t)\in (0,\infty) \times F$},\notag\\
				\mathrm{(ii)}&& |\partial_\lambda\theta_\lambda(x,t)|+|\partial_\lambda\tilde\theta_\lambda(x,t)| = |\partial_\lambda P_\lambda^\ast \varphi(x,t)|+|\partial_\lambda P_\lambda \tilde \varphi(x,t)|\leq \kappa_0 \mbox{\qquad if $(\lambda,x,t)\in \Omega$},\\
				\mathrm{(iii)}&& |\NT(\nabla_x\theta_\lambda)(x,t)|+|\NT(\nabla_x\tilde\theta_\lambda)(x,t)|\leq 2\kappa_0\mbox{\qquad if $(x,t)\in F$}. \end{eqnarray*}
	  \end{lem}
	  \begin{proof}
		  If $(x,t)\in F$, then by the fundamental theorem of calculus and the construction of the set $F$, see Definition~\ref{setF}~(iv),
		  \begin{eqnarray*}
		      |\theta_\lambda(x,t)| + |\tilde \theta_\lambda(x,t)| =
		      \int_0^{\lambda} |\partial_\sigma\P_{\sigma}^\ast\varphi(x,t)| + |\partial_\sigma \P_{\sigma}\tilde\varphi(x,t)| \, \d\sigma\leq \kappa_0\lambda.
		  \end{eqnarray*}
		  This proves (i). Similarly, consider $(\lambda,x,t)\in \Omega$. Then $(\lambda,x,t)\in \Gamma(x_0,t_0)$ for some $(x_0,t_0)\in F$ and since $\varphi$ and $\tilde \varphi$ are functions of $(x,t)$ only, we obtain
		  \begin{eqnarray*}
		      |\partial_\lambda\theta_\lambda(x,t)|=|\partial_\lambda \P_\lambda^\ast\varphi(x,t)|\leq N_\ast(\partial_\sigma \P_{\sigma}^\ast\varphi)(x_0,t_0)
		  \end{eqnarray*}
		together with an analogous estimate for $\partial_\lambda \tilde \theta_\lambda(x,t)$. Hence, (ii) is again a consequence of Definition~\ref{setF}~(iv). As $\varphi$ and $\tilde \varphi$ do not depend on $\lambda$, we also have
		\begin{eqnarray*}
		      |\NT(\nabla_x\theta_\lambda)(x,t)|+|\NT(\nabla_x\tilde\theta_\lambda)(x,t)|&\leq& \Max(|\nabla_x\varphi|^2)(x,t)^{1/2}+\Max(|\nabla_x\tilde\varphi|^2)(x,t)^{1/2}\notag\\
		      &&+|\NT(\nabla_x\P_{\lambda}^\ast\varphi)(x,t)|+|\NT(\nabla_x\P_{\lambda}\tilde\varphi)(x,t)|,
		\end{eqnarray*}
		showing that (iii) is a consequence of parts (i) and (v) in Definition~\ref{setF}.
	  \end{proof}

Our next lemma extends the bound in part (ii) above to the whole sawtooth region.

	  \begin{lem}\label{saw2}
		\begin{eqnarray*}\label{ra3karefined}
		      |\theta_\lambda(x,t)|+|\tilde\theta_\lambda(x,t)|\lesssim \kappa_0 \lambda\mbox{\qquad for $(\lambda,x,t)\in \Omega$}.
		\end{eqnarray*}
	  \end{lem}
	  \begin{proof}
		      By symmetry of the definitions it suffices to prove the bound for $\theta_{\lambda}$. As a preliminary observation note that if $(\lambda,x,t)\in \Omega$, then $(\lambda,x,t)\in \Gamma(x_0,t_0)$ for some $(x_0,t_0)\in F$ and in particular $(x,t)\in \Delta_{\lambda}(x_0,t_0)$. Since
		      $\varphi$ is a weak solution to the equation $\mathcal{H}_{\pa}^\ast\varphi=\div_{x}({A_{\no\ta}}\chi_{8\Delta})$ on $\R^{n+1}$, we can then use the classical local estimates for weak solutions with real coefficients, see e.g.\ Theorem 6.17 in~\cite{Li},  to the effect that
	      \begin{eqnarray*}
		      \sup_{(x,t)\in \Delta_{\lambda}(x_0,t_0)}|\varphi(x,t)-\varphi(x_0,t_0)|\lesssim \lambda + \bariint_{\Delta_{2\lambda}(x_0,t_0)}|\varphi(x,t)-\varphi(x_0,t_0)|\, \d x \d t.
	      \end{eqnarray*}
		      Hence, using the construction of the set $F$, see Definition~\ref{setF}~(iii), we deduce
	      \begin{eqnarray}\label{ra1intuu+}
	      \sup_{(x,t)\in \Delta_{\lambda}(x_0,t_0)}|\varphi(x,t)-\varphi(x_0,t_0)|\lesssim \lambda + \lambda \mathbb D \varphi(x_0,t_0) \leq \lambda+ \kappa_0 \lambda.
	      \end{eqnarray}
		
		      To start with the actual proof of the estimate stated in the lemma, we let $(\lambda,x,t)$ and $(x_0,t_0)$ be fixed as above and we denote by $\varphi_{2 \lambda}$ the average of $\varphi$ over the set $\Delta_{2\lambda}(x_0,t_0)$. Thinking of $P_\lambda^\ast$ as given by kernel representation from Lemma~\ref{resolvent kernel}, see also Remark~\ref{kernel on dotE}, we have $P_\lambda^\ast 1 = 1$ and consequently,
	    \begin{eqnarray*}
	    |\theta_\lambda(x,t)|=|(I-\P_{\lambda}^\ast)\varphi(x,t)|&\leq& |\varphi(x,t)-\varphi(x_0,t_0)|+|(I-\P_{\lambda}^\ast)\varphi(x_0,t_0)|\notag\\
	    &&+|\P_{\lambda}^\ast(\varphi-\varphi_{2\lambda})(x_0,t_0)|+|\P_{\lambda}^\ast(\varphi-\varphi_{2\lambda})(x,t)|.
	    \end{eqnarray*}
		      Hence, using \eqref{ra1intuu+} and Lemma~\ref{saw1}~(i), we deduce
	    \begin{eqnarray}\label{ra1intuu+ab}
	    |\theta_\lambda(x,t)|\lesssim \kappa_0 \lambda+|\P_{\lambda}^\ast(\varphi-\varphi_{2\lambda})(x_0,t_0)|+
	    |\P_{\lambda}^\ast(\varphi-\varphi_{2\lambda})(x,t)|.
	    \end{eqnarray}
	    To estimate the remaining two terms on the right, we bring in the kernel of $\P_{\lambda}^\ast$ explicitly. Indeed, Lemma~\ref{resolvent kernel} yields
	    \begin{eqnarray*}
	    \P_\lambda^\ast(\varphi-\varphi_{2\lambda})(x_0,t_0)=\iint_{\mathbb R^{n+1}}K_{\lambda, n+1}^\ast(x_0,t_0,z,\tau)(\varphi(z,\tau)-\varphi_{2\lambda})\, \d z\d \tau
	    \end{eqnarray*}
	    with a kernel enjoying the bound
	    \begin{align*}
	    |K^\ast_{\lambda,n+1} (x_0,t_0,z,\tau)| &\leq 1_{(0,\infty)}(\tau-t_{0}) \frac{C|t_0-\tau|^{-n/2+n+1-1}}{\lambda^{2+2n}}   \e^{- \frac{|t_0 - \tau|}{\lambda^2}} \e^{-c \frac{|x_0 - z|^2}{|t_0-\tau|}}
	    \\ &\lesssim 1_{(0,\infty)}(\tau-t_{0}) \frac{C}{\lambda^{n+2}}   \e^{- \frac{|t_0 - \tau|}{\lambda^2}} \e^{-c \frac{|x_0 - z|^2}{|t_0-\tau|}}
	    \end{align*}
	    for some constants $C,c>0$ depending only on dimension and ellipticity. So, splitting $\R^{n+2}$ into  $\Delta_{2 \lambda}(x_0,t_0)$ and annuli $\Delta_{2^{j+1} \lambda}(x_0,t_0) \setminus \Delta_{2^j \lambda}(x_0,t_0)$, $j \geq 1$, we can infer that
	    \begin{eqnarray}\label{ra1intuu+abcd}
	      |\P_{\lambda}^\ast(\varphi-\varphi_{2\lambda})(x_0,t_0)| \lesssim \sum_{j=0}^\infty 2^{(n+2)j} \e^{-c 4^j} \bariint_{\Delta_{2^{j+1} \lambda}(x_0,t_0)}|\varphi(z, \tau)-\varphi_{2 \lambda}|\, \d z\d \tau.
	    \end{eqnarray}
	    Next, by a telescopic sum and Lemma~\ref{Poincare} we deduce that
	    \begin{eqnarray*}
	    \bariint_{\Delta_{2^{j+1} \lambda}(x_0,t_0)}|\varphi(z, \tau)-\varphi_{2 \lambda}|\, \d z\d \tau
	    &\leq& \sum_{k=1}^{j+1} \bariint_{\Delta_{2^k \lambda}(x_0,t_0)} |\varphi(z, \tau)-\varphi_{2^k \lambda}|\, \d z\d \tau \\
	    &\leq& 2^{j+2} \lambda \biggl(\Max(|\gradx \varphi|)(x_0,t_0) + \Max_x \Max_t(|\HT \dhalf \varphi|)(x_0,t_0) \biggr)	
	    \end{eqnarray*}
	    and parts (i) and (ii) of Definition~\ref{setF} guarantee that the last term is no larger than $2^{j+3} \lambda \kappa_0$. In particular, summing up in \eqref{ra1intuu+abcd}, we can conclude $|\P_{\lambda}^\ast(\varphi-\varphi_{2\lambda})(x_0,t_0)|\lesssim \kappa_0 \lambda$. The estimate of $|\P_{\lambda}^\ast(\varphi-\varphi_{2\lambda})(x,t)|$ can be done similarly, taking into account $\Delta_{\lambda}(x_0,t_0) \subset \Delta_{2\lambda}(x,t)$ when writing out the telescopic sum of averages. Now, the claim follows from \eqref{ra1intuu+ab}.
	  \end{proof}

\subsection{An adapted cut-off and associated Carleson measures}{\label{Subsec:Psi}}
Here, we bring into play the degree of freedom $0<\eta \ll 1$ and the parameter $0<\epsilon \ll r$ that already appeared in the outline of Section~\ref{proof}.

Writing $\Gamma^\eta(x_0,t_0)$ for the parabolic cone with vertex $(x_0,t_0)$ and aperture $\eta$, we define the thinner sawtooth domains
		      \begin{eqnarray*}
		      \Omega_{\eta}:=\bigcup_{(x_0,t_0)\in F\cap \Delta}\Gamma^{\eta/8}(x_0,t_0).
		      \end{eqnarray*}
		      Then $\Omega_\eta \subset\Omega$. We are now going to define a smooth cut off adapted to $\Omega_\eta$.
		
		      Let $\Phi\in \C_0^\infty(\mathbb R)$ be such that $\Phi(\varrho)=1$ if $\varrho\leq 1/16$ and $\Phi(\varrho)=0$ if $\varrho>1/8$ and let $\Upsilon \in \C_0^\infty(\R^{n+2})$ be supported in $B(0,1/2048) \subset \R^{n+2}$ and satisfy $0 \leq \Upsilon \leq 1$ and $\int_{\R^{n+2}} \Upsilon(\mu,y,s) \, \d \mu \d y \d s = 1$. We then set
					\begin{equation}\label{cutoff}
			\begin{split}
			\Psi(\lambda, x,t) &:=\Psi_{\eta,\epsilon}(\lambda, x,t) \\
			&:=\Phi\biggl(\frac{\lambda}{32r}\biggr)\biggl(1-\Phi\biggl(\frac{\lambda}{32\epsilon}\biggr)\biggr) \int_{\Omega_{\eta/2}} \Upsilon\bigg(\frac{\lambda - \mu}{\lambda}, \frac{x-y}{\lambda \eta}, \frac{s-t}{(\lambda \eta)^2} \biggr) \, \frac{\d \mu \d y \d s}{\lambda (\lambda \eta)^{n+2}}.
			\end{split}
			\end{equation}
		      By construction, we have $\Psi\in \C^\infty_{0}(\R^{n+2}_{+})$, $0 \leq \Psi \leq 1$, $\Psi=1$ on the open set $\Omega_{\eta/4} \cap ((2 \epsilon,2r) \times \R^{n+1})$ and in particular on $(F\cap \Delta)\times (2\epsilon,2r)$ and the support of $\Psi$ is a subset of $\Omega_\eta \cap ((\epsilon,4r)\times 2\Delta)$. Making the link with the sawtooth domain $\Omega$ from the previous section, we note $$(\lambda,x,t) \in \supp \Psi \quad \Longrightarrow \quad (\eta \lambda,x,t) \in \Omega.$$
		      Now, let $\delta(x,t)$ denote the parabolic distance from the point $(x,t)\in\mathbb R^{n+1}$ to $F\subset \mathbb R^{n+1}$ and let
			  \begin{equation}\label{eq7+aha+g}
			  \begin{split}
      E_1&:=\{(\lambda, x,t)\in (0,4r)\times 2 \Delta:\ \eta\lambda/32\leq\delta(x,t)\leq \eta\lambda/8\},\\
      E_2&:=\{(\lambda,x,t) \in (2r,4r)\times 2\Delta: \delta(x,t) \leq \eta \lambda/8\},\\
      E_3&:=\{(\lambda,x,t) \in (\epsilon,2\epsilon)\times 2\Delta: \delta(x,t) \leq \eta \lambda/8\}.
			  \end{split}
			  \end{equation}
	Let $(\alpha,\beta,\gamma)=(\alpha,\beta_1,...,\beta_n,\gamma)\in \mathbb N^{n+2} \setminus \{0\}$ be a multiindex. Again by construction of $\Psi$ there exists a constant $\tilde c$ depending only on $\alpha$, $\beta$, $\gamma$, $\eta$ and $n$ such that
	\begin{eqnarray}\label{Linftyb}
	\biggl|\frac {\partial^{\alpha+|\beta|+\gamma}}{\partial \lambda^{\alpha}\partial x^{\beta}\partial {t^{\gamma}}}\Psi(\lambda, x,t)\biggr|
	&\leq& \frac{\tilde c}{|\lambda|^{\alpha + |\beta| + 2\gamma}} 1_{E_1 \cup E_2 \cup E_3}(\lambda,x,t).
	\end{eqnarray}

	The following Carleson lemma is also important in  the next section.
			\begin{lem}\label{tech} It holds
			\begin{eqnarray*}
	\iiint_{E_1 \cup E_2 \cup E_3}\, \frac{\d \lambda \d x \ d t}{\lambda} \leq \log(8) 2^{n+2} |\Delta|.
			\end{eqnarray*}
			In particular, let $\eta$, $\epsilon$ and $\Psi=\Psi_{\eta,\epsilon}$ be as above, let $(\alpha,\beta,\gamma)=(\alpha,\beta_1,...,\beta_n,\gamma)\in \mathbb N^{n+2} \setminus \{0\}$ be a multiindex and let $p \in (0,\infty)$. Then there exists $\tilde c = \tilde c(\alpha,\beta,\gamma,p, n,\eta) <\infty$, such that
			\begin{eqnarray*}
	\iiint_{\R^{n+2}_+}\, \biggl |\frac {\partial^{\alpha+|\beta|+\gamma}}{\partial \lambda^{\alpha}\partial x^{\beta}\partial {t^{\gamma}}}\Psi(\lambda, x,t)\biggr |^p\, \lambda^{p(\alpha+|\beta|+2\gamma)-1}\d \lambda\d x\d t \leq \tilde c |\Delta|.		
			\end{eqnarray*}			
			\end{lem}
			\begin{proof} By definition of the sets $E_1$, $E_2$ and $E_3$, integration in $(x,t)$ takes place only on the cube $2 \Delta$ and for $(x,t) \in 2\Delta$ fixed, integration in $\lambda$ is at most over the intervals $(\frac{8\delta(x,t)}{\eta}, \frac{32 \delta(x,t)}{\eta})$, $(2r, 4r)$ and $(\epsilon, 2 \epsilon)$, yielding a total contribution of $\log(8)$. Hence, the first claim follows from Tonelli's theorem and then the second one is a consequence of \eqref{Linftyb}
\end{proof}

\section{Proof of the Key Lemma}\label{sec6}
We are now ready to prove  the Key Lemma, hence completing the proof of Theorem~\ref{Ainfty+}. As discussed in Section~\ref{proof}, throughout the proof we can \emph{qualitatively} assume that $A$ is smooth. In that case, one can see that qualitatively $\varphi, \tilde \varphi, P_{\lambda}\tilde \varphi$ and $P_{\lambda}^\ast \varphi$ as well as $u$ are smooth by interior parabolic regularity. Furthermore, we will simply write $J$ for $J_{\eta,\epsilon}$ and we note -- and this is a consequence of the introduction of $\epsilon$ -- that no boundary terms will survive when we perform partial integration. Similarly, we will write $\Psi$ for $\Psi_{\eta,\epsilon}$ defined in Section~\ref{Subsec:Psi}. Throughout, $\sigma$ will denote a positive degree of freedom and $c$ will denote a generic constant (not necessarily the same at each instance), which depends only on the dimension $n$ and the ellipticity constants. In contrast, $\tilde c$ will denote a generic constant that may additionally depend on $\sigma$ and $\eta$.  The fact that $|u|\leq 1$  will be used repeatedly in the proof.

To start the estimate of $J$ we first note that $u\Psi^2\lambda$ is a test function for the weak formulation of the equation for $u$. Hence,
\begin{eqnarray}\label{est1}
0=\iiint_{\mathbb R^{n+2}_+} A\nabla_{\lambda,x} u\cdot\nabla_{\lambda,x}(u\Psi^2\lambda)+\partial_tu(u\Psi^2\lambda) \, \d x\d t\d\lambda.
\end{eqnarray}
As
\begin{eqnarray*}
\nabla_{\lambda,x} (u\Psi^2\lambda)=(\nabla_{\lambda,x} u)\Psi^2\lambda+u\lambda\nabla_{\lambda,x}\Psi^2+u\Psi^2\nabla_{\lambda,x}\lambda,
\end{eqnarray*}
we have
\begin{eqnarray*}
J&=&\iiint_{\mathbb R^{n+2}_+}A\nabla_{\lambda,x} u\cdot\nabla_{\lambda,x}(u\Psi^2\lambda)\, \d x\d t\d\lambda-\iiint_{\mathbb R^{n+2}_+}(A\nabla_{\lambda,x} u\cdot\nabla_{\lambda,x}\Psi^2)u\, \lambda \d x\d t\d\lambda\notag\\
&&-\iiint_{\mathbb R^{n+2}_+}(A\nabla_{\lambda,x} u\cdot\nabla_{\lambda,x}\lambda)u\Psi^2\, \d x\d t\d\lambda.
\end{eqnarray*}
Combining this with \eqref{est1}, we see that $J=\I_1+\I_2+\I_3$, where
  \begin{eqnarray*}
    \I_1&:=&-\iiint_{\mathbb R^{n+2}_+}(A\nabla_{\lambda,x} u\cdot\nabla_{\lambda,x}\Psi^2)u\, \lambda \d x\d t\d\lambda,\notag\\
    \I_2&:=&-\iiint_{\mathbb R^{n+2}_+}(A\nabla_{\lambda,x} u\cdot\nabla_{\lambda,x}\lambda)u\Psi^2\, \d x\d t\d\lambda,\notag\\
    \I_3&:=&-\iiint_{\mathbb R^{n+2}_+}\partial_tu(u\Psi^2\lambda)\, \d x\d t\d\lambda.
\end{eqnarray*}
The estimates of $\I_1$ and $\I_3$ turn out to be straightforward: Indeed, by the Cauchy-Schwarz inequality
\begin{eqnarray*}
  |\I_1| \leq c \bigg(\iiint_{\mathbb R^{n+2}_+} |\gradlamx u|^2 \Psi^2 \, \lambda \d x \d t \d \lambda\bigg)^{1/2} \bigg(\iiint_{\mathbb R^{n+2}_+} |\gradlamx \Psi|^2  \, \lambda \d x \d t \d \lambda\bigg)^{1/2}
\end{eqnarray*}
and hence, using the elementary Young's inequality, ellipticity of $A$ and the Carleson measure estimates in Lemma~\ref{tech},
	  \begin{eqnarray*}
	    |\I_1|\leq \sigma J+\tilde c|\Delta|.
	  \end{eqnarray*}
	Furthermore,
	  \begin{eqnarray*}
	\I_3=-\frac 1 2\iiint_{\mathbb R^{n+2}_+}(\partial_tu^2)\Psi^2\lambda\, \d x\d t\d\lambda=
	  \frac 1 2\iiint_{\mathbb R^{n+2}_+}u^2\partial_t(\Psi^2)\lambda\, \d x\d t\d\lambda.
	  \end{eqnarray*}
	Thus, by the Carleson measure estimates in Lemma~\ref{tech} and as $|u|, |\Psi| \leq 1$,
	  \begin{eqnarray*}
	|\I_3|\leq c\iiint_{\mathbb R^{n+2}_+}|\partial_t\Psi|\lambda\, \d x\d t\d\lambda\leq\tilde c|\Delta|.
	  \end{eqnarray*}
	As for $\I_2$, we first use the decomposition \eqref{eq:A} of the coefficients and split $\I_2=\I_{21}+\I_{22}$, where
	  \begin{eqnarray*}
	\I_{21}&:=&-\iiint_{\mathbb R^{n+2}_+}\bigl (A_{\no\pa}\cdot\nabla_{x}u\bigr )u\Psi^2\, \d x\d t\d\lambda,\notag\\
	\I_{22}&:=&-\iiint_{\mathbb R^{n+2}_+}\bigl (A_{\no\no}\partial_\lambda u\bigr )u\Psi^2\, \d x\d t\d\lambda.
	  \end{eqnarray*}
	  Since $A$ does not depend on $\lambda$, integration by parts yields
	      \begin{eqnarray*}
	\I_{22}=-\frac 1 2\iiint_{\mathbb R^{n+2}_+}A_{\no\no}\partial_\lambda u^2\Psi^2\, \d x\d t\d\lambda=\frac 1 2\iiint_{\mathbb R^{n+2}_+}A_{\no\no} 					u^2\partial_\lambda\Psi^2\, \d x\d t\d\lambda
	  \end{eqnarray*}
	  and hence $|\I_{22}|\leq \tilde c|\Delta|$ follows again by Lemma~\ref{tech}. To estimate $\I_{21}$ we use that
	  $$A_{\no\pa}\cdot\nabla_{x}\biggl (\frac {u^2\Psi^2}2\biggr )=(A_{\no\pa}\cdot\nabla_{x}u)u\Psi^2+(A_{\no\pa}\cdot\nabla_{x}\Psi) u^2\Psi$$
	  and we write $\I_{21}=\I_{211}+\I_{212}$, where
	      \begin{eqnarray*}
	\I_{211}&:=&-\iiint_{\mathbb R^{n+2}_+}A_{\no\pa}\cdot\nabla_{x}\biggl (\frac {u^2\Psi^2}2\biggr )\, \d x\d t\d\lambda,\notag\\
	\I_{212}&:=&\iiint_{\mathbb R^{n+2}_+}(A_{\no\pa}\cdot\nabla_{x}\Psi) u^2\Psi\, \d x\d t\d\lambda.
	      \end{eqnarray*}
	Once again, $|\I_{212}|\leq \tilde c|\Delta|$ follows by Lemma~\ref{tech}. In order to handle $\I_{211}$, we introduce $\varphi$ as in \eqref{eq7+int}, that is, as the energy solution on $\ree$ to the problem
	      \begin{eqnarray*}
	    \div_{x}(A_{\no\pa}\chi_{8\Delta})=-\partial_t\varphi-\div_{x}(A_{\pa\pa}^\ast\nabla_{x}\varphi)=\mathcal{H}_{\pa}^\ast\varphi.
	      \end{eqnarray*}
	The weak formulation with $\phi = u^2 \Psi^2(\lambda, \cdot, \cdot)/2$ as test function for $\lambda >0$ fixed, which by construction of $\Psi$ is supported in $8 \Delta$, yields
	      \begin{eqnarray*}
	\I_{211}= \int_{\R_+} \iint_{\ree} \varphi \partial_t \biggl (\frac {u^2\Psi^2}2\biggr )\, \d x\d t\d\lambda
	+\int_{\R_+} \iint_{\ree} A_{\pa\pa}^\ast\nabla_{x}\varphi \cdot \nabla_{x}\biggl (\frac {u^2\Psi^2}2\biggr )\, \d x\d t\d\lambda.
	      \end{eqnarray*}
	Recall that we write $\theta_{\eta \lambda} =\varphi-P_{\eta\lambda}^\ast\varphi$. Then, splitting $\varphi = \theta_{\eta \lambda} + P_{\eta\lambda}^\ast\varphi$ in both integrals, we may write
		  \begin{eqnarray} \label{I211}
	\I_{211}=\I_{2111}+\I_{2112}+\I_{2113}+\I_{2114},
		  \end{eqnarray}
	where
		  \begin{eqnarray*}
	\I_{2111}&:=&\iiint_{\mathbb R^{n+2}_+} (\theta_{\eta \lambda}) \partial_t \biggl (\frac {u^2\Psi^2}2\biggr )\, \d x\d t\d\lambda,\notag\\
	\I_{2112}&:=&\iiint_{\mathbb R^{n+2}_+} (P_{\eta\lambda}^\ast\varphi) \partial_t\biggl (\frac {u^2\Psi^2}2\biggr )\, \d x\d t\d\lambda,\notag\\
	\I_{2113}&:=&\iiint_{\mathbb R^{n+2}_+}A_{\pa\pa}^\ast\nabla_{x}\theta_{\eta \lambda} \cdot\nabla_{x}\biggl (\frac {u^2\Psi^2}2\biggr )\, \d x\d t\d\lambda,\notag\\
	\I_{2114}&:=&\iiint_{\mathbb R^{n+2}_+}A_{\pa\pa}^\ast\nabla_{x}P_{\eta\lambda}^\ast\varphi\cdot\nabla_{x}\biggl (\frac {u^2\Psi^2}2\biggr )\, \d x\d t\d\lambda.
	\end{eqnarray*}

	For the time being, let us concentrate on the second and fourth term in \eqref{I211}. Integrating by parts with respect to $\lambda$ leads us to
		    \begin{eqnarray*}
	\I_{2112}+\I_{2114}&=& -\iiint_{\mathbb R^{n+2}_+}(\partial_\lambda P_{\eta\lambda}^\ast\varphi) \partial_t\biggl (\frac {u^2\Psi^2}2\biggr )\, \lambda \d x\d t\d\lambda\notag\\
		    &&-\iiint_{\mathbb R^{n+2}_+}(P_{\eta\lambda}^\ast\varphi)\partial_t \partial_\lambda \biggl (\frac {u^2\Psi^2}2\biggr )\, \lambda \d x\d t\d\lambda\notag\\
		    &&-\iiint_{\mathbb R^{n+2}_+} (A_{\pa\pa}^\ast\nabla_{x}\partial_\lambda P_{\eta\lambda}^\ast\varphi)\cdot \nabla_{x} \biggl (\frac {u^2\Psi^2}2\biggr )\, \lambda \d x\d t\d\lambda\notag\\
		    &&-\iiint_{\mathbb R^{n+2}_+}(A_{\pa\pa}^\ast\nabla_{x}P_{\eta\lambda}^\ast\varphi)\cdot \nabla_{x} \partial_\lambda\biggl (\frac {u^2\Psi^2}2\biggr )\, \lambda \d x\d t\d\lambda,
		    \end{eqnarray*}
	where we have again used the $\lambda$-independence of the coefficients. We stress that throughout (and with a slight abuse of notation) $\partial_\lambda P_{\eta\lambda}^\ast\varphi$ denotes the derivative in $\lambda$ of the function $\lambda \mapsto P_{\eta\lambda}^\ast\varphi$, so that there is a factor $\eta$ showing up in front by the chain rule. A similar notational convention will apply to $\partial_\lambda \theta_{\eta \lambda}$.
	Taking into account the definition of the parabolic operator $\cH_\pa^*$, we can regroup these terms as
		  \begin{eqnarray*}
		    \I_{2112}+\I_{2114}=\T_1+\T_2+\T_3,
		  \end{eqnarray*}
	where
		    \begin{eqnarray*}
	  \T_1&:=&-\iiint_{\mathbb R^{n+2}_+}(\partial_\lambda P_{\eta\lambda}^\ast\varphi)\partial_t\biggl (\frac {u^2\Psi^2}2\biggr )\, \lambda \d x\d t\d\lambda,\notag\\
	  \T_2&:=&-\iiint_{\mathbb R^{n+2}_+}(A_{\pa\pa}^\ast\nabla_{x}\partial_\lambda P_{\eta\lambda}^\ast\varphi)\cdot \nabla_{x}\biggl (\frac {u^2\Psi^2}2\biggr )\, \lambda \d x\d t\d\lambda,\notag\\
	  \T_3&:=&\iiint_{\mathbb R^{n+2}_+}(\mathcal{H}^\ast_{\pa}P_{\eta\lambda}^\ast\varphi)\partial_\lambda\biggl (\frac {u^2\Psi^2}2\biggr )\, \lambda \d x\d t\d\lambda.
		    \end{eqnarray*}
	By the Cauchy-Schwarz inequality and the square function estimates stated in parts (ii) and (iii) of Lemma~\ref{square est} we first deduce that
			\begin{eqnarray*}
	|\T_2|+ |\T_3| &\leq& \tilde{c} |\Delta|^{1/2}\biggl (\iiint_{\mathbb R^{n+2}_+}\biggl |\nabla_{\lambda,x}\biggl (\frac {u^2\Psi^2}2\biggr )\biggr |^2\, \lambda \d x\d t\d\lambda\biggr)^{1/2} \\
			&\leq& \tilde c |\Delta|^{1/2} \biggl(\iiint_{\mathbb R^{n+2}_+}|\nabla_{\lambda,x} u|^2 \Psi + |\nabla_{\lambda, x} \Psi|^2 \, \lambda \d x\d t\d\lambda\biggr)^{1/2}
			\end{eqnarray*}
	and then, by Lemma~\ref{tech} and Young's inequality, we can conclude $|\T_2|+|\T_3|\leq  \sigma J+\tilde c|\Delta|$.
	To estimate $\T_1$, we write $\T_1=\T_{11}+\T_{12}$, where
	    \begin{eqnarray*}
		    \T_{11}&:=&-\iiint_{\mathbb R^{n+2}_+}(\partial_\lambda P_{\eta\lambda}^\ast\varphi)u^2\partial_t\biggl (\frac {\Psi^2}2\biggr )\, \lambda \d x\d t\d\lambda,\notag\\
		    \T_{12}&:=&-\iiint_{\mathbb R^{n+2}_+}(\partial_\lambda P_{\eta\lambda}^\ast\varphi)u\partial_tu {\Psi^2}\, \lambda \d x\d t\d\lambda.
		    \end{eqnarray*}
		    By a familiar argument relying on Cauchy-Schwarz, Lemma~\ref{square est} and Lemma~\ref{tech} we deduce $|\T_{11}|\leq \tilde c|\Delta|$. The estimate of $\T_{12}$ is more involved. Here, we first use the equation $\partial_t u = \div_{\lambda,x} A \gradlamx u$, which thanks to our smoothness assumption may be interpreted in the classical (pointwise) sense, in order to split $\T_{12}=\T_{121}+\T_{122}$, where
			\begin{eqnarray*}
		    \T_{121}&:=&-\iiint_{\mathbb R^{n+2}_+}(\partial_\lambda P_{\eta\lambda}^\ast\varphi)u\div_{x}(A\nabla_{\lambda,x} u)_{\pa}{\Psi^2}\, \lambda \d x\d t\d\lambda,\notag\\
		    \T_{122}&:=&-\iiint_{\mathbb R^{n+2}_+}(\partial_\lambda P_{\eta\lambda}^\ast\varphi)u (A\nabla_{\lambda,x} \partial_\lambda u)_{\no}{\Psi^2}\, \lambda \d x\d t\d\lambda.
		    \end{eqnarray*}
			  Then,
			\begin{eqnarray*}
		    \T_{121}&=&\iiint_{\mathbb R^{n+2}_+}\nabla_x(\partial_\lambda P_{\eta\lambda}^\ast\varphi) u \cdot (A\nabla_{\lambda,x} u)_{\pa}{\Psi^2}\, \lambda \d x\d t\d\lambda\notag\\
		    &&+\iiint_{\mathbb R^{n+2}_+}(\partial_\lambda P_{\eta\lambda}^\ast\varphi)\nabla_x u \cdot (A\nabla_{\lambda,x} u)_{\pa}{\Psi^2}\, \lambda \d x\d t\d\lambda\notag\\
		    &&+\iiint_{\mathbb R^{n+2}_+}(\partial_\lambda P_{\eta\lambda}^\ast\varphi)u(A\nabla_{\lambda,x} u)_{\pa}
		    \cdot \nabla_x{\Psi^2}\, \lambda \d x\d t\d\lambda.
			\end{eqnarray*}
		    For the first term on the right-hand side we can infer control by $\sigma J + \tilde{c} |\Delta|$ using Cauchy-Schwarz, Lemma~\ref{square est} and Young's inequality in a by now familiar manner. For the other two terms we shall use for the first time the definition of the set $F$. More precisely, in virtue of Lemma~\ref{saw1} we can replace the resolvent by its pointwise upper bound $|\partial_\lambda P_{\eta\lambda}^\ast\varphi(x,t)| \leq c \eta \leq c$ noting that $(\lambda,x,t) \in \supp(\Psi)$ implies $(\eta \lambda,x,t) \in \Omega$ by construction, see Section~\ref{Subsec:Psi}. Having done this, the second integral on the right-hand side is bounded by $\eta J$ thanks to ellipticity of $A$ and for the third one we obtain a bound $\tilde{c} J^{1/2} |\Delta|^{1/2}$ by applying Cauchy-Schwarz and Lemma~\ref{tech}. Put together, we have
			    \begin{eqnarray*}
		    |\T_{121}|\leq(\sigma+c\eta)J+\tilde c|\Delta|.
		    \end{eqnarray*}
		    Also, using Lemma~\ref{square est} we immediately have
			    \begin{eqnarray}\label{T1221}
		    |\T_{122}|\leq\sigma \T_{1221}+\tilde c|\Delta|,
		    \end{eqnarray}
		    where
			  \begin{eqnarray*}
		  \T_{1221}:= \iiint_{\mathbb R^{n+2}_+}|\nabla_{\lambda,x} \partial_\lambda u|^2{\Psi^2}\, \lambda^3 \d x\d t\d\lambda.
		    \end{eqnarray*}
		    This term can be estimated using a Whitney type covering argument, the fact that $\partial_\lambda u$ is a solution and Caccioppoli's inequality: Indeed, let $\mathcal{W}=\{W_i\}$ denote a partitioning of $\mathbb R^{n+2}_+$ into (parabolic) Whitney cubes, that is, each $W_i$ has dyadic (parabolic) sidelength $\ell(W_i)$ and is located at distance $4\ell(W_i)$ to the boundary. Let $\phi_i\in \C_0^\infty(2W_i)$ be a standard cut-off for $W_i$ such that $0\leq\phi_i\leq 1$, $|\nabla_{\lambda,x} \phi_i| + |\partial_t \phi_i|^{1/2} \leq c/ \ell(W_i)$ and $\sum_i\phi_i^2(\lambda,x,t)=1$ for all $(\lambda,x,t)\in \mathbb R^{n+2}_+$. Then
		    \begin{eqnarray*}
		      \T_{1221}& = &
		    \sum_i\iiint_{\mathbb R^{n+2}_+}|\nabla_{\lambda,x} \partial_\lambda u|^2{\phi_i^2\Psi^2}\, \lambda^3 \d x\d t\d\lambda\notag\\
		    &\leq& c \sum_i\iiint_{\mathbb R^{n+2}_+}|\partial_\lambda u|^2{|\nabla_{\lambda,x} (\phi_i\Psi)|^2}\, \lambda^3 \d x\d t\d\lambda\\
&&+ c \sum_i\iiint_{\mathbb R^{n+2}_+}|\partial_\lambda u|^2|\phi_i\Psi|{|\pd_t(\phi_i\Psi)|}\, \lambda^3 \d x\d t\d\lambda,
		    \end{eqnarray*}
		    by an application of  Lemma~\ref{lem:Caccioppoli} and hence, taking into account the finite overlap of the Whitney cubes,
			      \begin{eqnarray*}
		      \T_{1221}&\leq& c \iiint_{\mathbb R^{n+2}_+}|\partial_\lambda u|^2{\Psi^2}\, \lambda \d x\d t\d\lambda+ c\iiint_{\mathbb R^{n+2}_+}|\partial_\lambda u|^2{|\nabla_{\lambda,x} \Psi|^2}\, \lambda^3 \d x\d t\d\lambda\notag\\
&&+c\iiint_{\mathbb R^{n+2}_+}|\partial_\lambda u|^2{|\partial_t\Psi|}\, \lambda^3 \d x\d t\d\lambda.
		    \end{eqnarray*}
		    Crudely employing ellipticity, the first integral on the right-hand side is under control by $c J$. Since $\lambda|\partial_\lambda u|\leq c$ in a pointwise fashion, as follows easily from DeGiorgi--Moser--Nash interior estimates, Caccioppoli's inequality and $|u| \leq 1$, we can apply Lemma~\ref{tech} to bound the second and third integral by $\tilde{c} |\Delta|$. So, as to \eqref{T1221}, we have

		    \begin{eqnarray*}
		    |\T_{122}|\leq\sigma J+\tilde c|\Delta|.
		    \end{eqnarray*}
	Put together we can conclude that the second and fourth term all the way back in \eqref{I211} can be estimated by
			    \begin{eqnarray*}
		    |\I_{2112}+\I_{2114}|\leq \sigma J+\tilde c|\Delta|.
	\end{eqnarray*}

	At this stage of the proof it only remains to focus on $\I_{2111}+\I_{2113}$ and we note that by definition
		      \begin{eqnarray*}
		  \I_{2111}+\I_{2113}&=&-\iiint_{\mathbb R^{n+2}_+}\theta_{\eta \lambda} \partial_t \biggl(\frac {u^2\Psi^2}2 \biggr)\, \d x\d t\d\lambda+\iiint_{\mathbb R^{n+2}_+}A_{\pa\pa}^\ast\nabla_{x}\theta_{\eta \lambda} \cdot\nabla_{x}\biggl(\frac {u^2\Psi^2}2 \biggr)\, \d x\d t\d\lambda\notag\\
		    &=&\tildeT_1+\tildeT_2+\tildeT_3,
		      \end{eqnarray*}
		    where
		      \begin{eqnarray*}
		  \tildeT_1&:=&\iiint_{\mathbb R^{n+2}_+}\partial_tu({\theta_{\eta \lambda} u\Psi^2} )\, \d x\d t\d\lambda+\iiint_{\mathbb R^{n+2}_+} \nabla_{x}\theta_{\eta \lambda} \cdot A_{\pa\pa} \nabla_{x}u(u\Psi^2)\, \d x\d t\d\lambda,\notag\\
		  \tildeT_2&:=&\frac 1 2\iiint_{\mathbb R^{n+2}_+}\theta_{\eta \lambda} u^2\partial_t\Psi^2\, \d x\d t\d\lambda,\notag\\
		  \tildeT_3&:=&\frac 1 2\iiint_{\mathbb R^{n+2}_+} \nabla_{x}\theta_{\eta \lambda} \cdot A_{\pa\pa}\nabla_{x}\Psi^2 (u^2)\, \d x\d t\d\lambda.
		      \end{eqnarray*}
		    Using Lemma~\ref{saw2}, we have $|\theta_{\eta \lambda}(x,t)|\leq c \eta\lambda \leq c \lambda$ for $(\lambda,x,t)$ in the support of $\Psi$ and hence we can conclude, using Lemma~\ref{tech}, that $|\tildeT_2|\leq \tilde c|\Delta|$ holds. Similarly, for $\tildeT_3$ we would like to bring into play the (integrated) non-tangential control for $\nabla_{x}\theta_{\lambda}$ provided by Lemma~\ref{saw1}~(iii). To this end, we use an ``averaging trick'' justified by Tonelli's theorem in order to write
		      \begin{eqnarray*}
		\tildeT_3
		  \leq \tilde{c} \iiint_{\mathbb R^{n+2}_+} \biggl(\bariiint_{W_{\eta/64}(\sigma,y,s)}|\nabla_{x}\theta_{\eta \lambda} \cdot A_{\pa\pa} \nabla_{x}\Psi^2 (u^2)| \, \d \lambda \d x \d t \biggr) \, \d \sigma \d y \d s,
		      \end{eqnarray*}
		    where $W_{\eta/64}(\sigma,y,s)$ denotes the Whitney region $(\sigma/2, \sigma) \times Q_{\eta \sigma /64}(y) \times I_{\eta \sigma/64}(s)$. Now, letting
		      \begin{eqnarray*}
			\tilde{E}_1 &:=&\{(\sigma, y,s)\in (0,4r)\times 2 \Delta:\ \eta\sigma/64\leq\delta(y,s)\leq \eta\sigma/4\}, \\
			\tilde{E}_2 &:=&\{(\sigma, y,s) \in (2r,8r) \times 4 \Delta: \delta(y,s)\leq \eta\sigma/4\}, \\
			\tilde{E}_3 &:=&\{(\sigma, y,s) \in (\epsilon,4\epsilon) \times 4 \Delta: \delta(y,s)\leq \eta\sigma/4\},
		      \end{eqnarray*}
		    where again $\delta(x,t)$ denotes the parabolic distance from $(x,t)$ to the set $F$, we deduce from \eqref{eq7+aha+g} and \eqref{Linftyb} that the integrand on the right-hand side vanishes outside of $\tilde{E}_1 \cup \tilde{E}_2 \cup \tilde{E_3}$ and that we have a bound
		      \begin{eqnarray*}
		\tildeT_3
		  &\leq& \tilde{c} \iiint_{\tilde{E}_1 \cup \tilde{E}_2 \cup \tilde{E}_3} \biggl(\bariiint_{W_{\eta/64}(\sigma,y,s)}|\nabla_{x}\theta_{\eta \lambda}| \, \d \lambda \d x \d t \biggr) \, \frac{\d \sigma \d y \d s}{\sigma} \\
		  &\leq& \tilde{c} \iiint_{\tilde{E}_1 \cup \tilde{E}_2 \cup \tilde{E}_3} \biggl(\bariiint_{W_{\eta/64}(\eta \sigma,y,s)}|\nabla_{x} \theta_{\lambda}|^2 \, \d \lambda \d x \d t \biggr)^{1/2} \, \frac{\d \sigma \d y \d s}{\sigma},
		      \end{eqnarray*}
		    the second step following from Cauchy-Schwarz and a simple change of variables. The definition of $\tilde{E}_1 \cup \tilde{E}_2 \cup \tilde{E}_3$ entails that for $(\sigma,y,s)$ in this union, the Whitney region $W_{\eta/64}(\eta \sigma,y,s)$ is contained in a cone $\Gamma^{1/2}(x_0,t_0)$ with vertex $(x_0,t_0) \in F$. In particular, $W_{\eta/64}(\eta \sigma,y,s)$ can be covered by a finite number (depending only on $n$) of Whitney regions $\Lambda \times Q \times I$ showing up in the definition of the integrated maximal function $\NT$ on the set $F$. Consequently, Lemma~\ref{saw1} yields
		      \begin{eqnarray*}
		  \tildeT_3 \leq \tilde{c} \iiint_{\R^{n+2}_+} 1_{\tilde{E}_1 \cup \tilde{E}_2 \cup \tilde{E}_3}(\sigma,y,s) \, \frac{\d \sigma \d y \d s}{\sigma}.
		      \end{eqnarray*}
		    Up to a change of parameters,  the sets $\tilde{E}_1, \tilde{E}_2, \tilde{E}_3$ are similar to $E_1, E_{2}, E_3$ defined Subsection~\ref{Subsec:Psi} and so we can rely on Lemma~\ref{tech} to conclude $\tildeT_3 \leq \tilde{c} |\Delta|$. To estimate $\tildeT_1$, we start out with the identity
		      \begin{eqnarray*}
		    \nabla_{x}(\theta_{\eta \lambda} u\Psi^2)=(\nabla_{x}\theta_{\eta \lambda}) u\Psi^2+\theta_{\eta \lambda}(\nabla_{x}u)\Psi^2+\theta_{\eta \lambda} u\nabla_{x}(\Psi^2)
		      \end{eqnarray*}
		    to see that $ \tildeT_1= \tildeT_{11}+ \tildeT_{12}$, where
		      \begin{eqnarray*}
		    \tildeT_{11}&:=&\iiint_{\mathbb R^{n+2}_+}\partial_tu(\theta_{\eta \lambda} {u\Psi^2} )\, \d x\d t\d\lambda+\iiint_{\mathbb R^{n+2}_+} \nabla_{x}(\theta_{\eta \lambda} u\Psi^2)\cdot A_{\pa\pa} \nabla_{x}u\, \d x\d t\d\lambda,\notag\\
		    \tildeT_{12}&:=&-\iiint_{\mathbb R^{n+2}_+} \nabla_{x}u\cdot A_{\pa\pa} \nabla_{x}u(\theta_{\eta \lambda}\Psi^2)\, \d x\d t\d\lambda\notag-\iiint_{\mathbb R^{n+2}_+} \nabla_{x}\Psi^2\cdot A_{\pa\pa} \nabla_{x}u(u\theta_{\eta \lambda})\, \d x\d t\d\lambda.
		      \end{eqnarray*}
		    Using again the fact that $|\theta_{\eta \lambda}|\leq c \eta\lambda$ holds on the support of $\Psi$ along with Cauchy-Schwarz and Lemma~\ref{tech}, we deduce the estimate
		      \begin{eqnarray*}
		    |\tildeT_{12}|\leq c \eta J + \tilde{c} |\Delta|^{1/2} J^{1/2} \leq (\sigma+c\eta)J+\tilde c|\Delta|.
		      \end{eqnarray*}
		    To estimate $\tildeT_{11}$, we capitalize again that the smoothness of our coefficients allows us to plug in the equation $\partial_t u = \div_{\lambda,x} A \gradlamx u$ in the pointwise sense. Then, splitting $A$ according to \eqref{eq:A}, we can write $\tildeT_{11}=\tildeT_{111}+\tildeT_{112}+\tildeT_{113}$, where
			    \begin{eqnarray*}
			\tildeT_{111}&:=&-\iiint_{\mathbb R^{n+2}_+}A_{\pa\no}\cdot\nabla_{x}(\theta_{\eta \lambda} u\Psi^2)\partial_\lambda u\, \d x\d t\d\lambda,\notag\\
			\tildeT_{112}&:=&-\iiint_{\mathbb R^{n+2}_+}A_{\no\pa}\cdot\nabla_{x}u\partial_\lambda (\theta_{\eta \lambda} u\Psi^2)\, \d x\d t\d\lambda,\notag\\
			\tildeT_{113}&:=&-\iiint_{\mathbb R^{n+2}_+}A_{\no\no}\partial_\lambda(\theta_{\eta \lambda} u\Psi^2)\partial_\lambda u\, \d x\d t\d\lambda.
			    \end{eqnarray*}
			    Unwinding the derivative in $\lambda$ and using once more the bound $|\theta_{\eta \lambda}|\leq c \eta\lambda$ on the support of $\Psi$,
			    \begin{eqnarray*}
			      |\tildeT_{112}+\tildeT_{113}| \leq c \iiint_{\mathbb R^{n+2}_+} \eta \lambda |\gradlamx u (\partial_\lambda u) \Psi^2| + \lambda |\gradlamx u \partial_\lambda \Psi^2| + |\gradlamx u \partial_\lambda \theta_{\eta \lambda}| \, \d x \d t \d \lambda.
			    \end{eqnarray*}
			    Here, the first term gives a contribution $c \eta |\Delta|$, the second one can be treated by the familiar combination of Young's inequality and Lemma~\ref{tech}, whereas for the third term we make use of Lemma~\ref{square est}~(i) instead, noting that $\partial_\lambda \theta_{\eta \lambda} = \partial_\lambda P_{\eta \lambda}^\ast \varphi$ holds since $\varphi$ does not depend on $\lambda$. By these means, we find
			    \begin{eqnarray*}
			|\tildeT_{112}+\tildeT_{113}|\leq (\sigma+c\eta)J+\tilde c|\Delta|.
			    \end{eqnarray*}
			    Similarly, we obtain
			    $$|\tildeT_{111}-\tildeT_{1111}|\leq c\eta J+\tilde c|\Delta|,$$
			    where
			    $$\tildeT_{1111}:=\iiint_{\mathbb R^{n+2}_+}A_{\pa\no}\cdot\nabla_{x}\theta_{\eta \lambda} u\Psi^2\partial_\lambda u\, \d x\d t\d\lambda.$$
			    To estimate $\tildeT_{1111}$ we first integrate by parts in $\lambda$ and regroup derivatives to find
					\begin{eqnarray*}
			\tildeT_{1111}&=&\frac 1 2\iiint_{\mathbb R^{n+2}_+}A_{\pa\no}\cdot\nabla_{x}\theta_{\eta \lambda} \Psi^2\partial_\lambda u^2\, \d x\d t\d\lambda\notag\\
					&=&-\frac 1 2\iiint_{\mathbb R^{n+2}_+}A_{\pa\no}\cdot\nabla_{x}\theta_{\eta \lambda} \partial_\lambda \Psi^2 (u^2)\, \d x\d t\d\lambda\notag\\
					&&-\frac 1 2\iiint_{\mathbb R^{n+2}_+}A_{\pa\no}\cdot\nabla_{x}(\partial_\lambda P_{\eta\lambda}^\ast\varphi\Psi^2 u^2)\, \d x\d t\d\lambda \\
					&&+ \frac 1 2\iiint_{\mathbb R^{n+2}_+}A_{\pa\no} \cdot \partial_\lambda P_{\eta\lambda}^\ast\varphi \gradx(\Psi^2 u^2)\, \d x\d t\d\lambda.
					\end{eqnarray*}
			    Note that the first term on the right-hand side has the same structure as $\tildeT_3$ with the only exception that we have a $\lambda$-derivative on $\Psi$ instead of an $x$-derivative. Hence, we can derive a bound $\tilde{c} |\Delta|$ by the very same methods. Also the third term on the right-hand side is of the same kind as a term we encountered earlier in the proof -- $\T_2$ in this case -- which we already know how to bound by $\sigma J + \tilde{c} |\Delta|$.

			    All in all, we have reached a stage of the proof, where the only term that remains to be estimated is
				\begin{eqnarray*}
			\hatT_1:=\iiint_{\mathbb R^{n+2}_+}A_{\pa\no}\cdot\nabla_{x}(\Psi^2 u^2\partial_\lambda P_{\eta\lambda}^\ast\varphi)\, \d x\d t\d\lambda
				\end{eqnarray*}
			  and we remark that  this final term resembles $\I_{211}$ except that we have an additional factor $\partial_\lambda  P_{\eta\lambda}^\ast\varphi$ acting to our favor.
			    We now introduce $\tilde \varphi$ as in \eqref{eq7+int}, that is, as the energy solution to the problem
				\begin{eqnarray*}
				\div_{x}(A_{\pa\no}\chi_{8\Delta})=\partial_t\tilde\varphi-\div_{x}(A_{\pa\pa}\nabla_{x}\tilde \varphi)=\mathcal{H}_{\pa}\tilde\varphi
				\end{eqnarray*}
			    on $\ree$. We remark that                             $\Psi^2 u^2 \partial_\lambda P_{\eta\lambda}^\ast\varphi$ is qualitatively smooth and compactly supported, therefore it can  be used as test function for the equation above in order to rewrite $\hatT_1$. More precisely, we also recall $\tilde\theta_{\eta \lambda}=\tilde\varphi-\P_{\eta\lambda}\tilde\varphi$ and write
				\begin{eqnarray}
				\label{hatT_1}
					\hatT_1 = \hatT_{11}+\hatT_{12}+\hatT_{13},
				\end{eqnarray}
			    where
				\begin{eqnarray*}
			    \hatT_{11}&:=&-\iiint_{\mathbb R^{n+2}_+}\partial_t\tilde\theta_{\eta \lambda}(\Psi^2 u^2\partial_\lambda P_{\eta\lambda}^\ast\varphi)\, \d x\d t\d\lambda,\notag\\
			    \hatT_{12}&:=&-\iiint_{\mathbb R^{n+2}_+}A_{\pa\pa}\nabla_{x}\tilde\theta_{\eta \lambda}\cdot\nabla_{x}(\Psi^2 u^2\partial_\lambda P_{\eta\lambda}^\ast\varphi)\, \d x\d t\d\lambda,\notag\\
			    \hatT_{13}&:=&-\iiint_{\mathbb R^{n+2}_+}\mathcal{H}_{\pa}P_{\eta\lambda}\tilde\varphi(\Psi^2 u^2\partial_\lambda P_{\eta\lambda}^\ast\varphi)\, \d x\d t\d\lambda.
				\end{eqnarray*}

			    The estimate $|\hatT_{13}|\leq \tilde c |\Delta|$ is a consequence of the square function estimate in Lemma~\ref{square est} (i) and (iii). To estimate $\hatT_{12}$, we write $\hatT_{12}=\hatT_{121}+\hatT_{122}+\hatT_{123}$, where
				    \begin{eqnarray*}
			    \hatT_{121}&:=&-\iiint_{\mathbb R^{n+2}_+}A_{\pa\pa}\nabla_{x}\tilde\theta_{\eta \lambda} \cdot\nabla_{x}(u^2)(\Psi^2\partial_\lambda P_{\eta\lambda}^\ast\varphi)\, \d x\d t\d\lambda,\notag\\
			    \hatT_{122}&:=&-\iiint_{\mathbb R^{n+2}_+}A_{\pa\pa}\nabla_{x}\tilde\theta_{\eta \lambda} \cdot\nabla_{x}(\Psi^2)( u^2\partial_\lambda P_{\eta\lambda}^\ast\varphi)\, \d x\d t\d\lambda,\notag\\
			    \hatT_{123}&:=&-\iiint_{\mathbb R^{n+2}_+}A_{\pa\pa}\nabla_{x}\tilde\theta_{\eta \lambda} \cdot\nabla_{x}\partial_\lambda P_{\eta\lambda}^\ast\varphi(u^2\Psi^2)\, \d x\d t\d\lambda. \end{eqnarray*}
			    Once having applied the pointwise bound $|\partial_\lambda P_{\eta\lambda}^\ast\varphi| \leq c \eta \leq c$ on the support of $\Psi$, see Lemma~\ref{saw2}~(iii), the estimate of $\hatT_{122}$ reduces to that of $\tildeT_3$ with $\tilde \theta_{\eta \lambda}$ in lieu of $\theta_{\eta \lambda}$. As the latter two functions share identical estimates, we can record $|\hatT_{122}|\leq \tilde c|\Delta|$. To estimate $\hatT_{121}$ we first note, using Cauchy-Schwarz' and Young's inequality, that $$|\hatT_{121}|\leq c\hatT_{1211}+\sigma J,$$ where
				      \begin{eqnarray*}
			    \hatT_{1211}:= \iiint_{\mathbb R^{n+2}_+}\Psi^2|\partial_\lambda P_{\eta\lambda}^\ast\varphi|^2|\nabla_{x}\tilde\theta_{\eta \lambda}|^2\, \frac{\d x\d t\d\lambda}\lambda. \end{eqnarray*}
			    Now, by the averaging trick already used in the estimate of $\tildeT_3$,
				    \begin{eqnarray*}
			    \hatT_{1211}&\leq& \tilde c\iiint_{\mathbb R^{n+2}_+}\biggl (\bariiint_{W_{\eta/64}(\sigma,y,s)}\biggl(\Psi^2|\partial_\lambda P_{\eta\lambda}^\ast\varphi|^2|\nabla_{x}\tilde\theta_{\eta \lambda}|^2\biggr)\, {\d x\d t\d\lambda}\biggr )\frac{\d y\d s \d\sigma}\sigma\notag\\
				    &\leq& \tilde c\iiint_{\mathbb R^{n+2}_+}\biggl (\sup_{(\lambda, x,t)\in W_{\eta/64}(\sigma, y,s)}|\partial_\lambda P_{\eta\lambda}^\ast\varphi(x,t)|\biggr )^2\frac{\d y\d s \d\sigma}\sigma,
				    \end{eqnarray*}
				    where the second step follows again by Lemma~\ref{saw2}~(iii) and elementary geometric considerations as in the estimate for $\tildeT_3$. As before, we write $W_{\eta/64}(\sigma,y,s):=\Lambda_\sigma\times Q_{\eta \sigma/64}(y)\times I_{\eta \sigma/64}(s)$. {From \eqref{reproducing bound} we obtain
				    \begin{eqnarray*}
				    \biggl (\sup_{(\lambda, x,t)\in W_{\eta/64}(\sigma,y,s)}|\partial_\lambda P_{\eta\lambda}^\ast\varphi(x,t)|\biggr )^2
				    &\leq& \tilde{c} \sum_{j=1}^\infty e^{-c4^j}\bariint_{2^{j+1}Q_{\sigma}(y)\times 4^{j+1}I_{\sigma}(s)}
					|\sigma \cH_\pa^\ast \tilde P_{\sigma}^\ast \varphi(x,t)|^2\, \d x\d t,\end{eqnarray*}
				    where $\tilde P_{\sigma}^\ast = (1+\sigma^2 \cH_\pa^\ast)^{-1}$. Hence, using an averaging trick in the $(x,t)$-variables only,
				    \begin{eqnarray*}
				    \hatT_{1211}&\leq& \tilde c\iiint_{\mathbb R^{n+2}_+}\sum_{j=1}^\infty e^{-c4^j}\bariint_{2^{j+1}Q_{\sigma}(y)\times 4^{j+1}I_{\sigma}(s)}
					|\sigma \cH_\pa^\ast \tilde P_{\sigma}^\ast\varphi(x,t)|^2\, \d x\d t\frac{\d y\d s \d\sigma}\sigma\notag\\
					&= &\tilde c \sum_{j=1}^\infty e^{-c4^j} \iiint_{\mathbb R^{n+2}_+}|\sigma \cH_\pa^\ast \tilde P_{\sigma}^\ast\varphi(x,t)|^2\, \frac{\d x\d t \d\sigma}\sigma\leq\tilde c |\Delta|,
				    \end{eqnarray*}
				    where the final estimate follows from Lemma~\ref{square est}.} So, we can conclude $|\hatT_{121}|\leq \sigma J+\tilde c|\Delta|$.  To estimate
				    $\hatT_{123}$ we use that $u$ is scalar-valued and write
					    \begin{eqnarray*}
			    \hatT_{123}
			    &=&\iiint_{\mathbb R^{n+2}_+} \gradx(\tilde\theta_{\eta\lambda} u^2\Psi^2) \cdot A_{\pa \pa}^\ast \nabla_{x}\partial_\lambda P_{\eta\lambda}^\ast\varphi \, \d x\d t\d\lambda\notag\\
			    &&+\iiint_{\mathbb R^{n+2}_+}\tilde\theta_{\eta \lambda} \nabla_{x}(u^2\Psi^2) \cdot A_{\pa\pa}^\ast\nabla_{x}\partial_\lambda P_{\eta\lambda}^\ast\varphi \, \d x\d t\d\lambda.
		\end{eqnarray*}
		Note that so far we have neglected $\hatT_{11}$ appearing in \eqref{hatT_1}. Now, we come back to this term and combine it with the first integral on the right-hand side above to obtain
			    \begin{eqnarray*}
			    \hatT_{11}+\hatT_{123}&=&-\iiint_{\mathbb R^{n+2}_+}\tilde\theta_{\eta \lambda} u^2\Psi^2 (\mathcal{H}_{\pa}^\ast\partial_\lambda P_{\eta\lambda}^\ast\varphi)\, \d x\d t\d\lambda\notag\\
			    &&+\iiint_{\mathbb R^{n+2}_+}\tilde\theta_{\eta \lambda} \nabla_{x}(u^2\Psi^2) \cdot A_{\pa\pa}^\ast\nabla_{x}\partial_\lambda P_{\eta\lambda}^\ast\varphi \, \d x\d t\d\lambda\notag\\
			    &&+\iiint_{\mathbb R^{n+2}_+}\tilde\theta_{\eta \lambda} \partial_t(\Psi^2 u^2)\partial_\lambda P_{\eta\lambda}^\ast\varphi\, \d x\d t\d\lambda.
			    \end{eqnarray*}
			    The first term on the right can be bounded by
				\begin{eqnarray*}
			    \biggl(\iiint_{\mathbb R^{n+2}_+}|\tilde\theta_{\eta \lambda}|^2\, \frac {\d x\d t\d\lambda}{\lambda^3}\biggr )^{1/2}
			    \biggl(\iiint_{\mathbb R^{n+2}_+}|\lambda^2 \mathcal{H}_{\pa}^\ast\partial_\lambda P_{\eta\lambda}^\ast\varphi|^2\, \frac{\d x\d t\d\lambda}{\lambda} \biggr )^{1/2},
			    \end{eqnarray*}
			    which in itself is bounded by $\tilde c|\Delta|$ by square function estimates, see Lemma~\ref{square est}~(iv) and Lemma~\ref{square1}.
			    As for the second term on the right, having applied the pointwise bound $|\tilde\theta_{\eta \lambda}|\leq c \eta \lambda \leq c \lambda$ on the support of $\Psi$, we are left with the task of estimating $\T_2$, which we have done before.
			
			    Altogether,
			      \begin{eqnarray*}
			    \barT_1:=\iiint_{\mathbb R^{n+2}_+}\tilde\theta_{\eta \lambda} \partial_t(\Psi^2 u^2)\partial_\lambda P_{\eta\lambda}^\ast\varphi\, \d x\d t\d\lambda
			    \end{eqnarray*}
			    is now the only term that remains to be estimated. It is instructive to observe that -- upon replacing $|\tilde \theta_{\eta \lambda}|$ by its pointwise upper bound $c \lambda$ on the support of $\Psi$ -- this is the same term as $2 \T_1$. Hence, we can follow the treatment of the latter almost \emph{verbatim}, using the pointwise bound whenever feasible in order to reduce matters to estimates that have already been completed. So, we shall only outline the differences in this argument: To start the estimate we write $\barT_1=\barT_{11}+\barT_{12}$, where
			      \begin{eqnarray*}
			    \barT_{11}&:=&\iiint_{\mathbb R^{n+2}_+}\tilde\theta_{\eta \lambda}\partial_t(\Psi^2)u^2\partial_\lambda P_{\eta\lambda}^\ast\varphi\, \d x\d t\d\lambda, \\
			    \barT_{12}&:=&2\iiint_{\mathbb R^{n+2}_+}\tilde\theta_{\eta \lambda} u\partial_t u\Psi^2\partial_\lambda P_{\eta\lambda}^\ast\varphi\, \d x\d t\d\lambda.
			    \end{eqnarray*}
			    The estimate for $\barT_{11}$ follows from that of $\T_{11}$.
			    To estimate $\barT_{12}$ we use the equation for $u$ and write $\barT_{12}=\barT_{121}+\barT_{122}$, where
					  \begin{eqnarray*}
			      \barT_{121}&:=&2\iiint_{\mathbb R^{n+2}_+}\tilde\theta_{\eta \lambda} u\div_x(A\nabla_{\lambda,x} u)_{\pa}\Psi^2\partial_\lambda P_{\eta\lambda}^\ast\varphi\, \d x\d t\d\lambda,\notag\\
			      \barT_{122}&:=&2\iiint_{\mathbb R^{n+2}_+}\tilde\theta_{\eta \lambda} u(A\nabla_{\lambda,x} \partial_\lambda u)_{\no}\Psi^2\partial_\lambda P_{\eta\lambda}^\ast\varphi\, \d x\d t\d\lambda.
					  \end{eqnarray*}
Again, the estimate of $\barT_{122}$ follows from the bound for its counterpart $\T_{122}$. Furthermore,
					  \begin{eqnarray*}
			      \barT_{121}&=&2\iiint_{\mathbb R^{n+2}_+}\tilde\theta_{\eta \lambda} u(A\nabla_{\lambda,x} u)_{\pa}\Psi^2\cdot \nabla_x\partial_\lambda P_{\eta\lambda}^\ast\varphi\, \d x\d t\d\lambda\\
				&&+2\iiint_{\mathbb R^{n+2}_+}\tilde \theta_{\eta \lambda} \nabla_xu \cdot (A\nabla_{\lambda,x} u)_{\pa}\Psi^2\partial_\lambda P_{\eta\lambda}^\ast\varphi\, \d x\d t\d\lambda\notag\\
				&&+2\iiint_{\mathbb R^{n+2}_+}\tilde\theta_{\eta \lambda} u(A\nabla_{\lambda,x} u)_{\pa} \cdot \nabla_x\Psi^2\partial_\lambda P_{\eta\lambda}^\ast\varphi\, \d x\d t\d\lambda\notag\\
				&&+2\iiint_{\mathbb R^{n+2}_+}\nabla_x\tilde\theta_{\eta \lambda} u(A\nabla_{\lambda,x} u)_{\pa}\Psi^2\partial_\lambda  P_{\eta\lambda}^\ast\varphi\, \d x\d t\d\lambda,
					  \end{eqnarray*}
				where the estimate for the first three terms follows from the bound for $\T_{121}$ as before. Eventually, the fourth term, which shows up since unlike $\lambda$ the functions $\tilde \theta_{\eta \lambda}$ does also depend on $x$, can be bounded by $J^{1/2} |\hatT_{1211}|^{1/2} \leq \sigma J + \tilde{c} |\Delta|$ using Cauchy-Schwarz and the previously obtained bound for $\hatT_{1211}$. Put together this completes the proof of the Key Lemma, Lemma~\ref{Carleson}.

\def\cprime{$'$} \def\cprime{$'$} \def\cprime{$'$}

\end{document}